\documentclass[reqno]{rmmcart_draft}

\usepackage{geometry}
\usepackage{paralist}
\usepackage[latin1]{inputenc}
\usepackage{graphicx}
\usepackage[format=hang,textfont={small,sl},labelfont={small,bf,sl},width=0.97\textwidth]{caption}
\usepackage[format=plain,textfont={footnotesize,sl},labelfont={footnotesize,bf,sl}]{subcaption}
\usepackage{amsmath,amssymb,amsfonts,amsthm}
\usepackage{bbm}
\usepackage[shortlabels]{enumitem}
\usepackage{bm}

\usepackage[usenames,dvipsnames]{xcolor}

\usepackage[norefs,nocites]{refcheck}
\usepackage{autonum}
\usepackage{./citesort}

\newtheorem{Ass}{Assumption}
\newtheorem{prop}{Proposition}

\DeclareMathOperator*{\argmin}{arg\;min}
\newcommand{\Tsup}{^{\text{\scriptsize T}}}
\newcommand{\msup}{^{\text{\scriptsize m}}}
\newcommand{\tsup}{^{\text{t}}}
\newcommand{\sym}{_{\text{sym}}}
\newcommand{\msub}{_{\text{\scriptsize m}}}
\newcommand{\oo}{\omega}
\newcommand{\sip} {\! \cdot\!}
\newcommand{\pdemi} {\mbox{$\frac{1}{2}$}}
\newcommand{\dip} {\! :\!}
\newcommand{\demi} {\inv{2}}
\newcommand{\CSs} {\CS^{\star}}
\newenvironment{cmatrix}{\left\{\begin{matrix}}{\end{matrix}\right\}}
\newcommand{\OO}{\Omega}
\newcommand{\G}{\Gamma}
\newcommand{\GN}{\Gamma_{\text{N}}}
\newcommand{\GD}{\Gamma_{\text{D}}}
\newcommand{\Gc}{\G_{c}}
\newcommand{\iO}{\int_{\OO}}
\newcommand{\dV}{\;\text{d}V}
\newcommand{\dO}{\partial\OO}
\newcommand{\Gu}{\G_{u}}
\newcommand{\Ecal}{\mathcal{E}}
\newcommand{\Dcal}{\mathcal{D}}
\newcommand{\CS}{\text{\boldmath $\mathcal{C}$}}
\newcommand{\Vcal}{\mathcal{V}}
\newcommand{\Ucal}{\mathcal{U}}
\newcommand{\Wcal}{\mathcal{W}}
\newcommand{\Scal}{\mathcal{S}}
\newcommand{\Qcal}{\mathcal{Q}}
\newcommand{\AS}{\text{\boldmath $\mathcal{A}$}}
\newcommand{\Lcal}{\mathcal{L}}
\newcommand{\Ccal}{\mathcal{C}}
\newcommand{\Fcal}{\mathcal{F}}
\newcommand{\Acal}{\mathcal{A}}
\newcommand{\Bcal}{\mathcal{B}}
\newcommand{\Hcal}{\mathcal{H}}
\newcommand{\Kcal}{\mathcal{K}}
\newcommand{\Zcal}{\mathcal{Z}}
\newcommand{\Ncal}{\mathcal{N}}
\newcommand{\Rbb} {\mathbb{R}}
\newcommand{\bfb} {\boldsymbol{b}}
\newcommand{\bfu} {\boldsymbol{u}}
\newcommand{\bfsig}{\boldsymbol{\sigma}}
\newcommand{\bfg} {\boldsymbol{g}}
\newcommand{\bfn} {\boldsymbol{n}}
\newcommand{\bfze}{\mathbf{0}}
\newcommand{\bfeps}{\boldsymbol{\varepsilon}}
\newcommand{\bfa} {\boldsymbol{a}}
\newcommand{\bfwT} {\widetilde{\boldsymbol{w}}{}}
\newcommand{\bff} {\boldsymbol{f}}
\newcommand{\bfw} {\boldsymbol{w}}
\newcommand{\bfv} {\boldsymbol{v}}
\newcommand{\bfx} {\boldsymbol{x}}
\newcommand{\bftau}{\boldsymbol{\tau}}
\newcommand{\bfsigT}{\tilde{\bfsig}}
\newcommand{\bfuT} {\widetilde{\boldsymbol{u}}{}}
\newcommand{\CSh} {\hat{\CS}}
\newcommand{\bfus} {\boldsymbol{u}^{\star}}
\newcommand{\bfws} {\boldsymbol{w}^{\star}}
\newcommand{\bfd} {\boldsymbol{d}}
\newcommand{\bfvT} {\widetilde{\boldsymbol{v}}{}}
\newcommand{\sfU} {\mathsf{U}}
\newcommand{\sfF} {\mathsf{F}}
\newcommand{\sfUT} {\widetilde{\sfU}}
\newcommand{\Ubb} {\mathbb{U}}
\newcommand{\sfG} {\mathsf{G}}
\newcommand{\bft} {\boldsymbol{t}}
\newcommand{\bfz} {\boldsymbol{z}}
\newcommand{\sfY} {\mathsf{Y}}
\newcommand{\bfps}{\boldsymbol{\psi}}
\newcommand{\Nbb} {\mathbb{N}}
\newcommand{\uhat}{\hat{u}}
\newcommand{\bfL} {\boldsymbol{L}}
\newcommand{\bfy} {\boldsymbol{y}}
\newcommand{\sfW} {\mathsf{W}}
\newcommand{\bfh} {\boldsymbol{h}}
\newcommand{\bfF} {\boldsymbol{F}}
\newcommand{\bfG} {\boldsymbol{G}}
\newcommand{\bfH} {\boldsymbol{H}}
\newcommand{\bfE} {\boldsymbol{E}}
\newcommand{\shsubs}{\hspace*{-0.1em}\subset\hspace*{-0.1em}}
\newcommand{\shleq}{\hspace*{-0.1em}\leq\hspace*{-0.1em}}
\newcommand{\Div}{\mbox{div}\,} 
\newcommand{\sheq}{\hspace*{-0.1em}=\hspace*{-0.1em}}
\newcommand{\shcap}{\hspace*{-0.1em}\cap\hspace*{-0.1em}}
\newcommand{\bfna}{\boldsymbol{\nabla}}
\newcommand{\shneq}{\hspace*{-0.1em}\not=\hspace*{-0.1em}}
\newcommand{\shdeq}{\hspace*{-0.1em}:=\hspace*{-0.1em}}
\newcommand{\shcup}{\hspace*{-0.1em}\cup\hspace*{-0.1em}}
\newcommand{\shsetm}{\hspace*{-0.1em}\setminus\hspace*{-0.1em}}
\newcommand{\lpar}{\big(\hspace*{0.1em}}
\newcommand{\rpar}{\hspace*{0.1em}\big)}
\newcommand{\lbra}{\big\langle\hspace*{0.1em}}
\newcommand{\rbra}{\hspace*{0.1em}\big\rangle}
\newcommand{\shin}{\hspace*{-0.1em}\in\hspace*{-0.1em}}
\newcommand{\lcb}{\big\{\hspace*{0.1em}}
\newcommand{\Lcb}{\Big\{\hspace*{0.1em}}
\newcommand{\Rcb}{\hspace*{0.1em}\Big\}}
\newcommand{\rcb}{\hspace*{0.1em}\big\}}
\newcommand{\shgeq}{\hspace*{-0.1em}\geq\hspace*{-0.1em}}
\newcommand{\inv}[1]{\dfrac{1}{#1}}
\newcommand{\shg}{\hspace*{-0.1em}>\hspace*{-0.1em}}
\newcommand{\shm}{\hspace*{-0.1em}-\hspace*{-0.1em}}
\newcommand{\shtimes}{\hspace*{-0.1em}\times\hspace*{-0.1em}}
\newcommand{\shp}{\hspace*{-0.1em}+\hspace*{-0.1em}}
\newcommand{\tens}{\hspace*{-1pt}\otimes\hspace*{-1pt}}
\newcommand{\shl}{\hspace*{-0.1em}<\hspace*{-0.1em}}
\newcommand{\nnum}{\notag \\}
\newcommand{\lsqb}{\big[\hspace*{0.1em}}
\newcommand{\rsqb}{\hspace*{0.1em}\big]}
\newcommand{\Lpar}{\Big(\,}
\newcommand{\Rpar}{\,\Big)}
\newcommand{\shoplus}{\hspace*{-0.1em}\oplus\hspace*{-0.1em}}
\newcommand{\suite}[1][0ex]{\notag \\[#1] & \mbox{}\hspace{15pt}}

\newcommand{\bfuH}{\widehat{\bfu}}
\newcommand{\Lambdak}{\Lambda_{\kappa}}
\newcommand{\LambdaRk}{\tilde{\Lambda}{}_{\kappa}}
\newcommand{\EcalR}{\tilde{\Ecal}}
\newcommand{\DcalR}{\tilde{\Dcal}}

\title[Analysis of ECE approach for elasticity imaging]{\Large Analysis of the error in constitutive equation approach for time-harmonic elasticity imaging}

\author{Wilkins Aquino}
\address{Dept. of Mech. Eng. and Mater. Sci., Duke University, Durham, USA}
\email{wilkins.aquino@duke.edu}
\author{Marc Bonnet}
\address{POEMS (CNRS, INRIA, ENSTA), ENSTA, 91120 Palaiseau, France. mbonnet@ensta.fr}
\email{mbonnet@ensta.fr}
\date{\today}
\begin{document}
\begin{abstract}
We consider the identification of heterogeneous linear elastic moduli in the context of time-harmonic elastodynamics. This inverse problem is formulated as the minimization of the modified error in constitutive equation (MECE), an energy-based cost functional defined as an weighted additive combination $\Ecal+\kappa\Dcal$ of the error in constitutive equation (ECE) $\Ecal$, expressed using an energy seminorm, and a quadratic error term $\Dcal$ incorporating the kinematical measurements. MECE-based identification are known from existing computational evidence to enjoy attractive properties such as improved convexity, robustness to resonant frequencies, and tolerance to incompletely specified boundary conditions (BCs). The main goal of this work is to develop theoretical foundations, in a continuous setting, allowing to explain and justify some of the aforementioned beneficial properties, in particular addressing the general case where BCs may be underspecified. A specific feature of MECE formulations is that forward and adjoint solutions are governed by a fully coupled system, whose mathematical properties play a fundamental role in the qualitative and computational aspects of MECE minimization. We prove that this system has a unique and stable solution at any frequency, provided data is abundant enough (in a sense made precise therein) to at least compensate for any missing information on BCs. As a result, our formulation leads in such situations to a well-defined solution even though the relevant forward problem is not \emph{a priori} clearly defined. This result has practical implications such as applicability of MECE to partial interior data (with important practical applications including ultrasound elastography), convergence of finite element discretizations and differentiability of the reduced MECE functional. In addition, we establish that usual least squares and pure ECE formulations are limiting cases of MECE formulations for small and large values of $\kappa$, respectively. For the latter case, which corresponds to exact enforcement of kinematic data, we furthermore show that the reduced MECE Hessian is asymptotically positive for any parameter perturbation supported on the measurement region, thereby corroborating existing computational evidence on convexity improvement brought by MECE functionals. Finally, numerical studies that support and illustrate our theoretical findings, including a parameter reconstruction example using interior data, are presented.
\end{abstract}
\maketitle

\section{Introduction}

Energy-based objective functionals are strong alternatives to conventional least square methods for various parameter identification problems. Such functionals, often called error in constitutive equation (ECE) functionals in the area of solid mechanics, were initially introduced in~\cite{ladeveze:83} for error estimation in the linear elastic FEM and in e.g.~\cite{kohn:lowe,knowles:01} for electrical impedance tomography. ECE functionals have since been successfully used in various mechanical parameter identification problems under linear static~\cite{gey:pag}, nonlinear quasistatic~\cite{marchand:18}, time-harmonic~\cite{lad:ned:rey:94,lad:chou:99,barthe:04,B-2012-2} and, more recently, transient conditions ~\cite{allix:05,feissel:allix:06,nguyen:08,B-2014-3}. Mathematical and numerical issues are also discussed in e.g.~\cite{chavent:96,gock:khan:09}.\enlargethispage*{3ex}

The main idea behind ECE approaches is to relax the constitutive equations connecting fluxes (e.g. stresses) and gradients of state variables (e.g strains). To this end, stresses (assumed to satisfy dynamic equilibrium) and strains (assumed to satisfy kinematical constraints) are treated as independent quantities. ECE functionals then measure residuals in constitutive equations evaluated on given stresses and strains, thereby assuming a physical meaning directly connected to constitutive parameter identification. In their original form, ECE functionals assume the measured displacements or strains to be strictly enforced as part of the kinematical admissibility constraints, but this is often undesirable as real data is usually polluted with noise. For this reason, ECE-based identification is nowadays rather formulated by means of so-called modified error in constitutive equation (MECE) functionals~\cite{lad:ned:rey:94}, where reliable and unreliable informations are treated differently. Equilibrium equations, initial conditions and known boundary conditions (BCs) are deemed reliable and enforced strictly (e.g. using Lagrange multipliers). By contrast, measured data, constitutive properties and (when applicable) imperfectly known BCs are deemed unreliable and incorporated as constitutive or observation residuals using an additive combination $\Ecal+\kappa\Dcal$ of ECE and least-squares components $\Ecal$ and $\Dcal$, with $\kappa$ a positive weight parameter.

Optimization problems that arise from MECE-based identification have been observed to display very attractive properties such as improved convexity, robustness to resonant frequencies, and tolerance to partially or completely unknown BCs. For instance, the MECE objective for transient elastodynamics was observed in~\cite{feissel:allix:06} to be convex over a wide region of the parameter space. Convexity improvement relative to least-squares functionals was also reported in \cite{B-2003-8} in the context of linear elastostatics. Moreover, for time-harmonic conditions, the FEM-discretized coupled stationarity problem was shown in \cite{B-2012-2} to remain uniquely solvable at resonant prescribed frequencies. In addition, as shown in~\cite{B-2015-1,B-2014-3}, ECE approaches can naturally accommodate configurations with partially or completely unknown BCs, a feature also used in~\cite{feissel:15} for parameter identification using elastostatic interior data (i.e. unknown BCs). This makes such formulations perfectly suited to situations where interior data is abundant over a subset of the medium being probed. Important practical applications include elasticity imaging~\cite{ghosh:17} where interior displacements are tracked inside soft tissue using ultrasound, but information about region boundaries is difficult to ascertain.

These advantages of MECE formulations over conventional least squares counterparts are backed by abundant numerical and experimental evidence, but to our best knowledge lack theoretical support in infinite-dimensional settings . Some of the existing analyses address discretized MECE formulations~\cite{B-2012-2,B-2014-3}, while \cite{gock:khan:09} studied existence of solutions and convexity of a continuous ECE formulation using complete internal data. Accordingly, the main goal of this work is to develop theoretical foundations allowing to explain and justify some of the aforementioned beneficial properties of MECE formulations. We focus on elasticity imaging under time-harmonic conditions and adopt a Hilbert space setting. Moreover, our analysis addresses the general case where BCs may be underdetermined.\enlargethispage*{5ex}

A specific feature of MECE formulations is that first-order optimality conditions lead to fully coupled forward and adjoint problems, rather than unidirectionally coupled problems arising in conventional least squares approaches. The mathematical properties of this coupled system play a fundamental role in the qualitative (e.g. existence of solutions, continuity and differentiability of the objective with respect to parameters) and numerical aspects of MECE-based imaging. In particular, our formulation leads to a well-defined  solution in the underspecified BC case, for which a relevant forward problem is not \emph{a priori} clearly defined. Hence, establishing the well-posedness of the coupled system is our first main goal. Treating this system as a perturbed mixed problem, we prove that it has a unique and stable solution at any frequency, subject to conditions that ensure that the available data more than compensates insufficient information on BCs. This result has important practical implications such as applicability of MECE to partial interior data, convergence of finite element discretizations and differentiability of the reduced MECE objective functional. Secondly, establish that least squares and pure ECE formulations are limiting forms for $\kappa\to0$ and $\kappa\to\infty$, respectively, of MECE formulations. Thirdly, for the latter case (corresponding to exact enforcement of kinematic data), we show that the reduced MECE hessian becomes positive for any parameter perturbation supported on the measurement region. The latter result also has strong practical implications as convexity translates into robustness with respect to the initial guess in gradient-based optimization.

This work is organized as follows. In Section~\ref{setting}, we define the elasticity imaging problem and set notations, introduce the relevant MECE functional and state the relevant first-order stationarity conditions. Section~\ref{sec:2} then addresses the well-posedness of the coupled stationarity problem, allowing for underdetermined BCs. Then, the limiting forms of the stationarity solution as $\kappa\to0$ or $\kappa\to\infty$ are established in Section~\ref{sec:3}. The reduced MECE Hessian, and in particular its asymptotic convexity in the $\kappa\to\infty$ limit, is studied in Sec.~\ref{sec:hessianAnalysis}. Then, numerical studies that support and illustrate our theoretical findings, including a parameter reconstruction example using interior data, are presented in Section~\ref{sec:num}. Finally, most of the proofs are deferred to Section~\ref{sec:proofs}.

\section{Problem setting}
\label{setting}

Let a solid elastic body occupy a bounded and connected domain $\OO\shsubs\Rbb^d$ $(1 \shleq d \shleq 3)$ with boundary $\G$. The time-harmonic motion of this body is governed by (i) the balance equations
\begin{equation}
  \Div\bfsig + \bfb = -\rho\oo^2\bfu \quad \text{in }\OO, \qquad \bfsig\sip\bfn = \bfg \quad \text{on }\GN, \label{balance}
\end{equation}
where $\bfu $ is the displacement field, $\oo$ represents the specified angular frequency, $\rho$ denotes the known mass density, $\bfb$ is a given body force density, $\bfsig$ represents the stress tensor, $\bfg$ and $\GN \subseteq \G$ are the given surface force density (traction) and its support, respectively, and $\bfn$ is the outward unit vector normal to $\G$; (ii) the kinematic compatibility equations
\begin{equation}
  \bfu = \bfze \quad \text{on } \GD, \qquad \bfeps[\bfu] = \pdemi(\bfna\bfu + \bfna\bfu\Tsup ) \quad \text{in } \OO, \label{compatibility}
\end{equation}
where $\bfeps[\bfu]$ denotes the linearized strain tensor associated with $\bfu$ and $\GD\subseteq\G$ is the portion of the boundary where the displacement is prescribed; and (iii) the (linear elastic) constitutive relation 
\begin{equation}
  \bfsig = \CS \dip \bfeps[\bfu] \quad \text{in } \OO,  \label{constitutive}
\end{equation}
where $\CS$ is the fourth-order elasticity tensor field. For simplicity, the kinematical boundary condition specified in~\eqref{compatibility} is homogeneous; the case of a non-homogeneous boundary condition can be treated with minor modifications, as can loading arrangements other than those appearing in~\eqref{balance}.

Here, the boundary subsets $\GN$ and $\GD$ are only required to not overlap (i.e., $\GN \shcap \GD \sheq \emptyset$) and may  cover $\G$ only partially (i.e., $\GN \shcup \GD \subseteq \G$). In other words, we may have $\Gc\shneq\emptyset$, where $\Gc\shdeq\G\shsetm(\GN \shcup \GD)$. In such event, equations \eqref{balance}-\eqref{constitutive} admit multiple solutions, whereas a unique solution exists (except for a countable set of eigenfrequencies $\oo$) when $\GN\cup\GD = \G$. This non-standard boundary condition setting is adopted here to model experimental situations where full-field interior data (to be introduced thereafter) is available while boundary conditions are underspecified.\enlargethispage*{5ex}

\paragraph{Measurements} In addition to the fundamental equations \eqref{balance}-\eqref{constitutive}, we assume availability for the prescribed frequency $\oo$ of measured time-harmonic displacements $\bfu\msup$ in $\OO\msub\shsubs\OO$ (interior data), which is for example usual in elastography applications~\cite{ghosh:17}. The forthcoming analyses can be adapted to accommodate measured displacements on a portion of $\G\shsetm\GD$.

\paragraph{Inverse problem} We address the inverse problem of reconstructing the elasticity tensor field $\CS$ such that (i) the governing equations of motion \eqref{balance}-\eqref{constitutive} are satisfied, and (ii) $\CS$ is consistent with the measurement $\bfu\msup$, under prescribed time-harmonic conditions. Equality $\bfu = \bfu\msup \text{ in } \OO\msub$ between the data $\bfu\msup$ and its model counterpart $\bfu$, implicit in requirement (ii), should hold for error-free model and measurements but will be relaxed in the upcoming formulation to allow for expected uncertainties.

\subsection{Weak formulation of motion}

Let $\lpar \bfa,\bfb \rpar$ denote the $L^2(\OO)$ scalar product of second-order tensor fields $\bfa,\bfb\shin L^2(\OO;\Rbb^{d\times d})$:\vspace*{-1ex}
\begin{equation}
  \lpar \bfa,\bfb \rpar := \iO \bfa \dip\overline{\bfb} \dV = \iO a_{ij} \overline{b}_{ij} \dV,
\end{equation}
where the overline means complex conjugation. Repeated indices imply summation wherever indicial notation is used. The $L^2(\OO)$ scalar product of vector or scalar fields follows the same notational style with suitable adjustments; so does the $L^2(\G)$ scalar product of fields defined on a surface $\G$, denoted as $\lpar \bfa,\bfb \rpar_\G$. The weak formulation of the balance equations~\eqref{balance} then reads
\begin{equation} \label{balance:weak}
  \lpar \bfsig, \bfeps[\bfwT] \rpar - \oo^2 \lpar \rho \bfu,\bfwT \rpar - \lpar \bfsig \sip \bfn, \bfwT \rpar_{\G\setminus\GN}
  = \lbra \bff, \bfwT \rbra_{\Vcal',\Vcal} \qquad \text{for all } \bfwT \in \Vcal\shdeq H^1(\OO;\Rbb^d),
\end{equation}
where the continuous linear functional $\bff\shin\Vcal'$ embodies all given excitations; for example $\lbra \bff, \bfwT \rbra_{\Vcal',\Vcal}:=\lpar \bfb, \bfw \rpar + \lpar \bfg, \bfw \rpar_{\GN}$ for the given force densities $\bfb,\bfg$ appearing in~\eqref{balance}. In~\eqref{balance:weak} and thereafter, $\langle \cdot,\cdot \rangle_{X',X}$ denotes the duality pairing between a Hilbert space $X$ an its dual $X'$ (i.e. $\lbra \bff,\bfwT \rbra_{X',X}$ evaluates the linear functional $\bff\shin X'$ at $\bfwT\shin X$). In addition, let the spaces, $\Ucal$, $\Scal(\bfu)$ and $\Qcal$ of kinematically admissible displacements, dynamically admissible stresses and admissible elasticity tensor fields, respectively, be defined as
\begin{subequations}
\begin{align}
  \Ucal &:= \lcb \bfv\shin \Vcal\sheq H^1(\OO;\Rbb^d), \; \bfv\sheq\bfze \text{ on } \GD \rcb \label{Ucal:def} \\
  \Scal(\bfu) &:= \big\{ \bfsig : \ \bfsig\in \Scal, \text{ eqn.~\eqref{balance:weak} holds} \big\}, \qquad \text{with}\;\;
  \Scal = \lcb \bfsig\shin L^2(\OO;\Rbb^{d\times d}\sym),\ \Div\bfsig\shin  L^2(\OO;\Rbb^d)\rcb \label{Scal:def} \\
  \Qcal &:= \Lcb \CS \in L^\infty(\OO;\mathbbm{Q}) \mid \, \bfeps\dip\CS(\bfx)\dip\bfeps \shgeq c_0 \bfeps\dip\bfeps, \ \ \text{for all }\bfx\shin\OO \text{ and } \bfeps\shin\Rbb^{d\times d}\sym, \, \bfeps \neq \bfze \Rcb \label{Z:def}
\end{align}
\end{subequations}
(where $\mathbbm{Q}$ denotes the finite-dimensional vector space of fourth-order tensors $\CS$ with major and minor symmetries, i.e. $\Ccal_{ijkl} = \Ccal_{klij} = \Ccal_{jilk}$, and $c_0$ is some positive constant). Finally, the mass density field $\rho\shin L^\infty(\OO)$ must be bounded below by a positive constant.\enlargethispage*{3ex}

\subsection{MECE optimization problem}

We follow the inversion approach initiated in~\cite{B-2012-2,B-2013-8}, whereby the foregoing inverse problem is formulated as an optimization problem in which the unknown elasticity tensor field $\CS$ is estimated by minimizing an objective function that additively combines two error terms: (i) an error in constitutive equation (ECE) functional~\cite{ladeveze:83} defined by
\begin{equation} \label{eqn:ECEdefinition}
  \Ecal(\bfu,\bfsig,\CS) := \demi\iO (\bfsig-\CS\dip\bfeps[\bfu])\dip\CS^{-1}\dip(\bfsig-\CS\dip\bfeps[\bfu]) \dV
\end{equation}
that measures (in units of energy) the discrepancy in the constitutive equation~\eqref{constitutive}, and (ii) a quadratic error term $\Dcal(\bfu\shm\bfu\msup,\bfu\shm\bfu\msup)$, where $\Dcal$ is a positive bilinear form, that quantifies the mismatch between the predicted (or model) displacements and the measured ones. This objective function, hereafter called the modified ECE (MECE) functional, is defined by
\begin{equation}\label{eqn:LambdaDefinition}
  \Lambdak(\bfu,\bfsig,\CS) := \Ecal(\bfu,\bfsig,\CS) + \kappa\Dcal(\bfu\shm\bfu\msup,\bfu\shm\bfu\msup),
\end{equation}
where $\kappa\shg0$ is a weight parameter. For a given triple $(\bfu,\bfsig,\CS) \in \Ucal \times \Scal(\bfu) \times \Qcal$ (with these sets as defined in~(\ref{Ucal:def},\ref{Scal:def},\ref{Z:def})), $\Lambdak(\bfu,\bfsig,\CS)$ defines a quantitative measure of the consistency of these variables with (i) the constitutive equation, and (ii) the available measurements $\bfu\msup$. Accordingly, the elasticity imaging inverse problem is formulated as the PDE-constrained optimization problem
\begin{equation}\label{eqn:InvProbDefinition}
  (\bfu,\bfsig,\CS) :=  \argmin_{\bfv\in\Ucal, \, \bftau\in\Scal(\bfv), \, \AS\in\Qcal} \, \Lambdak(\bfv,\bftau,\AS). 
\end{equation}

\subsection{Stationarity conditions}

We now collect the first-order stationarity conditions for the minimization problem~\eqref{eqn:InvProbDefinition}. To this end, let the Lagrangian $
\Lcal: \Ucal\shtimes\Vcal\shtimes\Scal\shtimes\Qcal \to \Rbb$ be defined as\enlargethispage*{7ex}
\begin{equation} \label{Lagr:def}
  \Lcal(\bfu,\bfw,\bfsig,\CS) := \Lambdak(\bfu,\bfsig,\CS)
 - \Re \lcb \lpar \bfsig, \bfeps[\bfw] \rpar - \oo^2 \lpar \rho \bfu,\bfw \rpar - \lpar \bfsig \sip \bfn, \bfw \rpar_{\G\setminus\GN} - \lbra \bff, \bfw \rbra_{\Vcal',\Vcal} \rcb,
\end{equation}
where (i) $\bfw \shin \Vcal$ plays the role of the Lagrange multiplier, (ii) the constraint is the dynamic balance equation~\eqref{balance:weak}, and (iii) the term $\lpar \bfsig\sip\bfn, \bfw \rpar_{\G\setminus\GN}$ is crucial for the case in which $\GN \shcup \GD \shneq \Gamma$ (i.e., boundary conditions are not prescribed over the entire boundary). The first-order optimality conditions for the minimization problem ~\eqref{eqn:InvProbDefinition} are then given, in terms of first-order G\^ateaux derivatives of $\Lcal$, by
\begin{subequations}
\begin{align}
  \partial_{\sigma}\Lcal(\bfu,\bfw,\bfsig,\CS)[\bfsigT] &= 0 \quad\text{for all }\bfsigT\shin\Scal, \label{eq:DsigL} \\
  \partial_w\Lcal(\bfu,\bfw,\bfsig,\CS)[\bfwT] &= 0 \quad\text{for all }\bfwT\shin\Vcal, \label{eq:DwL} \\
  \partial_u\Lcal(\bfu,\bfw,\bfsig,\CS)[\bfuT] &= 0 \quad\text{for all }\bfuT\shin\Ucal, \label{eq:DuL} \\
  \partial_{\CS}\Lcal(\bfu,\bfw,\bfsig,\CS)[\CSh] &= 0 \quad\text{for all }\CSh \shin\Qcal, \label{eq:DcsL} 
\end{align}
\end{subequations}
where the test functions $\bfuT$ are constrained, consistently with the assumption $\bfu\shin\Ucal$, whereas $\bfwT$ are for now unconstrained. We begin by exploiting conditions~\eqref{eq:DsigL} and~\eqref{eq:DwL}. Condition~\eqref{eq:DsigL} yields
\begin{equation}
  \Re \lcb \lpar \bfsigT, \CS^{-1}\dip\bfsig - \bfeps[\bfu\shp\bfw] \rpar \rcb
 + \Re \lcb \lpar \bfsigT\sip\bfn,\bfw \rpar_{\G\setminus\GN} \rcb = 0 \qquad\text{for all }\bfsigT\shin\Scal. \label{eq:DsigL:2}
\end{equation}
while condition~\eqref{eq:DwL} simply restates the balance constraint~\eqref{balance:weak}. Using first the subspace of $\Scal$ containing all $\bfsigT$ with vanishing trace on $\G$, equation~\eqref{eq:DsigL:2} provides
\begin{equation}
  \bfsig = \CS\dip\bfeps[\bfu\shp\bfw] \quad\text{in }\OO, \label{stat:sigma:1}
\end{equation}
The second term in~\eqref{eq:DsigL:2} then enforces the essential condition
\begin{equation}
  \bfw=\bfze \quad\text{on }\G\shsetm\GN, \label{stat:sigma:2}
\end{equation}
the stress-type unknown $\bfsig\sip\bfn\mid_{\G\setminus\GN}$ in the fourth term of~\eqref{balance:weak} being the associated Lagrange multiplier, see e.g.~\cite{babuska:73}. We elect in this work to treat condition~\eqref{stat:sigma:2} via elimination, and hence seek $\bfw$ in the function space $\Wcal$ defined by
\begin{equation}
  \Wcal = \lcb \bfv\shin\Vcal, \; \bfv\sheq\bfze \text{ on } \GD\shcup\Gc \rcb. \label{Wcal:def}
\end{equation}
For the same reason, we now restrict equation~\eqref{balance:weak} to test functions $\bfwT\shin\Wcal$, thereby also eliminating the Lagrange multiplier $\bfsig\bfn\mid_{\G\setminus\GN}$, and obtain, after using~(\ref{stat:sigma:1}):
\begin{equation}
  \lpar \bfeps[\bfu\shp\bfw], \CS\dip\bfeps[\bfwT] \rpar - \oo^2 \lpar \rho\bfu,\bfwT \rpar = \Fcal(\bfwT), \quad \text{for all } \bfwT\in\Wcal.  \label{stat:w}
\end{equation}

We note that $\Wcal \subseteq \Ucal$; more precisely, $\Wcal\shsubs\Ucal$ (with strict inclusion) when $\Gamma_N\shcup\Gamma_D\shneq\Gamma$ (insufficient boundary data), while $\Wcal\sheq\Ucal$ when $\Gamma_N\shcup\Gamma_D\sheq\Gamma$ (sufficient boundary data).

Then, moving to the third condition~\eqref{eq:DuL}, we find
\begin{equation}
  \partial_u\Lcal(\bfu,\bfw,\bfsig,\CS)[\bfuT]
 = \Re \lcb \lpar \bfeps[\bfu]-\CS^{-1}\dip\bfsig), \CS\dip\bfeps[\bfuT] \rpar + \kappa \Dcal\lpar \bfu\shm\bfu\msup, \bfuT \rpar
  + \oo^2\lpar \rho \bfuT, \bfw \rpar \rcb,
\end{equation}
so that, on substituting~\eqref{stat:sigma:1} and rearranging, \eqref{eq:DuL} yields
\begin{equation}
\lpar \bfeps[\bfw], \CS\dip\bfeps[\bfuT] \rpar - \oo^2\lpar \rho \bfuT, \bfw \rpar - \kappa \Dcal\lpar \bfu\shm\bfu\msup, \bfuT \rpar  = 0 \quad\text{for all }\bfuT\shin\Ucal, \label{stat:u}
\end{equation}
Finally, condition~\eqref{eq:DcsL} yields the following equation, which is nonlinear in $\CS$:
\begin{equation}
\lpar \bfeps[\bfu]\tens\bfeps[\bfu] - (\CS^{-1}\dip\bfsig) \tens (\CS^{-1}\dip\bfsig) \,,\, \CSh \rpar = 0
\qquad \text{for all }\CSh\shin\Qcal \label{stat:C}
\end{equation}

Concluding, the first-order stationarity conditions for the minimization problem~\eqref{eqn:InvProbDefinition} consist of the coupled weak equations~\eqref{stat:w}, \eqref{stat:u}, \eqref{stat:C} governing $(\bfu,\bfw,\CS)$, with $\bfsig$ then given explicitly by~\eqref{stat:sigma:1}.\enlargethispage*{3ex}

\subsection{Reduced optimization problem}

In a full-space approach, the stationarity system~(\ref{stat:sigma:1},\ref{stat:w},\ref{stat:u},\ref{stat:C}) is solved (iteratively) as a whole, as done e.g. in~\cite{epanomeritakis:08}. Alternatively, a reduced-space approach can be formulated from observing that equations~(\ref{stat:w},\ref{stat:u}) define for given $\CS$ a \emph{linear} (coupled) problem for $(\bfu,\bfw)$. Then, letting $(\bfu,\bfw)\sheq(\bfu_\Ccal,\bfw_\Ccal)$ solve equations~(\ref{stat:w},\ref{stat:u}) for a given $\CS \in \Qcal$, reduced versions of the ECE and data-misfit components of $\Lambdak$ that depend only on $\CS$ can be defined by
\begin{equation}
\EcalR(\CS) \shdeq \Ecal(\bfu,\bfsig,\CS) = \lpar \bfeps[\bfw],\CS\dip\bfeps[\bfw] \rpar, \qquad
\DcalR(\CS) \shdeq \Dcal(\bfu\shm\bfu\msup,\bfu\shm\bfu\msup), \label{ED:reduced}
\end{equation}
where we have used $\bfsig\sheq\CS\dip[\bfu\shp\bfw]$, see~\eqref{stat:sigma:1}. Then, problem~\eqref{eqn:InvProbDefinition} can be recast in reduced minimization form as
\begin{equation}
\min_{\CS\in\Qcal} \LambdaRk(\CS), \qquad
\LambdaRk(\CS) := \Lambdak(\bfu,\bfsig,\CS) = \EcalR(\CS) + \kappa \DcalR(\CS), \label{min:C}
\end{equation}
We observe that the functional $\Lambdak(\bfu,\bfsig,\CS)$ is, for any fixed $\CS\shin\Qcal$, differentiable and convex in $(\bfu,\bfsig)$ due to the requisite ellipticity of $\CS$ (see~\eqref{Z:def}). Problem~(\ref{stat:w},\ref{stat:u}) is hence equivalent to solving the partial minimization of $\Lambdak(\bfu,\bfsig,\CS)$ with $\CS$ given, and we have
\begin{equation}\label{eqn:MechanicalFieldsUpdate}
  \LambdaRk(\CS) = \argmin_{\bfu\in \Ucal,\, \bfsig\in\Scal(\bfu)} \, \Lambdak(\bfu,\bfsig,\CS).
\end{equation}
For that reason, we will henceforth refer to problem~(\ref{stat:w},\ref{stat:u}) as the stationarity system.

In this work, we adopt and study this reduced-space approach, which thanks to the characterization~\eqref{eqn:MechanicalFieldsUpdate} can in fact be shown to be equivalent to the full-space approach in the following sense:\enlargethispage*{5ex}
\begin{lemma}
Problems~\eqref{min:C} and~\eqref{eqn:InvProbDefinition} are equivalent: a solution to~\eqref{min:C} also solves~\eqref{eqn:InvProbDefinition} and vice versa.
\end{lemma}
\begin{proof}
First, let $\CSs$ solve the reduced problem \eqref{min:C}, and $(\bfus,\bfws)$ solve~(\ref{stat:w}-\ref{stat:u}) for $\CS\sheq\CSs$. Then:
\begin{equation}
  \LambdaRk(\CSs)
 = \Lambdak(\bfus,\bfws,\CSs) \leq \Lambdak(\bfu_\Ccal,\bfw_\Ccal,\CS) \leq \Lambdak(\bfu,\bfw,\CS)
\end{equation}
for any $(\bfu,\bfw,\CS)\in\Ucal\shtimes\Wcal\shtimes\Qcal$, where the first inequality stems from~\eqref{min:C} and the second from~\eqref{eqn:MechanicalFieldsUpdate}. Hence, $(\bfus,\bfws,\CSs)$ solves~\eqref{eqn:InvProbDefinition}.

Conversely, let $(\bfu^{\sharp},\bfw^{\sharp},\CS^{\sharp})$ solve problem~\eqref{eqn:InvProbDefinition}. This triple then verifies~(\ref{stat:w}-\ref{stat:u}) with $\CS\sheq\CS^{\sharp}$, i.e. has the form $(\bfu_{\Ccal^{\sharp}},\bfw_{\Ccal^{\sharp}},\CS^{\sharp})$, so that $\Lambdak(\bfu^{\sharp},\bfw^{\sharp},\CS^{\sharp})=\LambdaRk(\CS^{\sharp})$. Since in addition problem~\eqref{min:C} consists in minimizing $\Lambdak$ over a subset of $\Ucal\shtimes\Wcal\shtimes\Qcal$, $\CS^{\sharp}$ is a minimizer of $\LambdaRk$. The proof is complete.
\end{proof}

\subsection{Coupled stationarity system: a key component of MECE-based imaging}
\label{remarks:coupled}

The stationarity system~(\ref{stat:w},\ref{stat:u}) plays for several reasons a fundamental role in MECE-based imaging:
\begin{compactenum}[(a)]
\item The definition~\eqref{min:C} of a reduced-space approach needs well-posedness of the stationarity system;
\item The present MECE framework aims at treating situations with underspecified BCs, for which a relevant forward problem is not \emph{a priori} clearly defined, in contrast with usual PDE-constrained inversion methods. Instead, the field $\bfu$ acts as a forward solution, provided problem~(\ref{stat:w},\ref{stat:u}) is well-posed. It is in particular important to determine conditions on the interior data ensuring that it compensates the insufficient BC information and make problem~(\ref{stat:w},\ref{stat:u}) well-posed.
\item In cases involving underspecified BCs, the relevant forward problem is not \emph{a priori} clearly defined. Instead, the  stationarity system~(\ref{stat:w},\ref{stat:u}) acts as a combination of the forward and adjoint problems. The latter are coupled in the present framework, while they are usually uncoupled in inverse problems solved using standard $L^2$ minimization~\cite{epanomeritakis:08,metivier:12,hinze:08} (see Sec.~\ref{sec:small-k-obj} for further elaboration).
\item Well-posedness of the stationarity system implies that the solution mapping $\CS\mapsto(\bfu,\bfw)$ is Fr\'echet differentiable (by virtue of the implicit function theorem, e.g.~\cite[Thm.~7.13-1]{ciarlet:book}, whose applicability can then readily be verified). Hence, $\LambdaRk(\CS)$ is also differentiable in that case, since $\Lambdak(\bfu,\bfsig,\CS)$ is, upon expressing $\sigma$ with~\eqref{stat:sigma:1}, quadratic in $(\bfu,\bfw)$ and affine in $\CS$.
\item In turn, continuity of $\LambdaRk(\CS)$ is one of two key ingredients needed to establish existence of solutions to Problem ~\eqref{min:C}. The other key ingredient would be establishing the existence of minimizing sequences in $\Qcal$ that contain convergent subsequences; it is outside of the scope of this work. Interested readers are referred to \cite{gock:khan:09} for an in-depth study of existence of solutions in problems involving functionals similar to those studied herein.
\end{compactenum}
The above considerations show the importance of establishing the well-posednedess of the stationarity system~(\ref{stat:w},\ref{stat:u}); this is the goal of the next section.\enlargethispage*{5ex}

\section{Analysis of the stationarity problem}
\label{sec:2}

In preparation for the forthcoming analysis, we rewrite the stationarity problems ~(\ref{stat:u}-\ref{stat:w}) in the form
\begin{equation}
\text{Find }(\bfw,\bfu)\in\Wcal\shtimes\Ucal, \quad \left\{
\begin{aligned}
  \text{(a)} && \Acal(\bfw,\bfwT,\CS) + \Bcal(\bfu,\bfwT,\CS) &= \lbra \bff, \bfwT \rbra_{\Wcal',\Wcal} && \text{for all }\bfwT\shin\Wcal \\
  \text{(b)} && \Bcal(\bfuT,\bfw,\CS) - \kappa \Dcal(\bfu,\bfuT) &= -\kappa\lbra \bfd, \bfuT \rbra_{\Ucal',\Ucal}  && \text{for all }\bfuT\shin\Ucal
\end{aligned} \right. \label{stat:pb}
\end{equation}
where $\Acal$ is the elastic stiffness bilinear form given for any $(\bfv,\bfvT)\shin\Ucal\shtimes\Ucal$ by $\Acal(\bfv,\bfvT,\CS)\shdeq\lpar \bfeps[\bfv],\CS\dip\bfeps[\bfvT] \rpar$, $\Bcal\shdeq\Acal-\oo^2(\rho\cdot,\cdot)$ is the dynamic stiffness bilinear form, the linear functional $\bfd\shin\Ucal'$ defined by $\lbra \bfd, \bfuT \rbra_{\Ucal',\Ucal} = \Dcal(\bfu\msup,\bfuT)$ incorporates the kinematic data, the linear functional $\bff\in\Wcal'$ synthesizes all given applied loads, and the function spaces $\Ucal,\Wcal$ are respectively defined by~\eqref{Ucal:def} and~\eqref{Wcal:def}. Moreover, we endow $\Ucal$ and $\Wcal$ with the inner product $(\cdot,\cdot)_{\oo}$ and norm $\|\cdot\|_{\oo}$ defined by
\[
  (\bfv,\bfvT)_{\oo} = \Acal(\bfv,\bfvT) + \rho\oo^2(\bfv,\bfvT), \qquad \|\bfv\|_{\oo}^2 := (\bfv,\bfv)_{\oo}.
\]
The norm $\|\cdot\|_{\oo}$ is equivalent (for fixed $\oo\shg0$) to the standard $H^1$ norm (if $|\GD|\shg0$, the alternative definition  $\|\bfuT\|^2=\Acal(\bfuT,\bfuT)$ is also suitable). The bilinear forms $\Acal$, $\Bcal$ and $\Dcal$ are assumed to be continuous for this norm, i.e. there exist constants $a\shg0$, $b\shg0$ and $d\shg0$ such that
\[
  \Acal(\bfw,\bfwT) \leq a\,\|\bfw\|_{\oo}\|\bfwT\|_{\oo}, \quad \Bcal(\bfu,\bfwT) \leq b\,\|\bfu\|_{\oo}\|\bfwT\|_{\oo}, \quad \Dcal(\bfu,\bfuT) \leq d\,\|\bfu\|_{\oo}\|\bfuT\|_{\oo}
\]
for all $\bfw,\bfwT\in\Wcal$ and $\bfu,\bfuT\in\Ucal$. In what follows, we will denote by $A, B, D$ the linear operators implicitly defined by the bilinear forms $\Acal,\Bcal,\Dcal$, e.g.:
\[
  B: \Ucal\to\Wcal{}', \;\; \lbra B\bfu,\bfwT \rbra_{\Wcal',\Wcal} = \Bcal(\bfu,\bfwT) \;\; \text{for all } (\bfu,\bfwT)\shin\Ucal\shtimes\Wcal,
\]
Moreover, $B\tsup:\Wcal\to\Ucal'$ will denote the transposed operator associated to $B$, defined by
\[
  \lbra \bfuT,B\tsup\bfw \rbra_{\Ucal,\Ucal'} = \lbra B\bfuT,\bfw \rbra_{\Wcal',\Wcal} = \Bcal(\bfuT,\bfw) \qquad \text{for all } (\bfuT,\bfw)\shin\Ucal\shtimes\Wcal,
\]
and $N(F)$ will denotes the null space of a linear operator $F$.

\subsection{Well-posedness of the coupled problem}

We regard problem~\eqref{stat:pb} as a perturbed mixed problem~\cite[Sec. 4.3]{boffi}. It can in fact be given the form
\[
  \text{Find $\sfU\shin\Ubb$ \ such that} \quad G(\sfU,\sfUT,\CS) = \lbra \sfF,\sfUT \rbra_{\Ubb',\Ubb} \qquad \text{for all } \sfUT\shin\Ubb,
\]
where $\Ubb$ is the Hilbert space $\Ubb\shdeq\Ucal\shtimes\Wcal$ equipped with the inner product defined by $(\sfU,\sfUT)_{\Ubb}=(\bfu,\bfuT)_{\oo}+(\bfw,\bfwT)_{\oo}$ (the associated norm being given by $\|\sfU\|^2_{\Ubb}=(\sfU,\sfU)_{\Ubb}$) and with the (continuous, symmetric) bilinear form $G:\Ubb\shtimes\Ubb\to\Rbb$ and the (continuous) linear form $\sfF\in\Ubb'$ defined by
\begin{align}
  G(\sfU,\sfUT,\CS) &:= \Acal(\bfw,\bfwT,\CS) + \Bcal(\bfu,\bfwT,\CS) + \Bcal(\bfuT,\bfw,\CS) - \kappa \Dcal(\bfu,\bfuT), \\
  \lbra \sfF,\sfUT \rbra_{\Ubb',\Ubb} &:= \lbra \bff,\bfwT \rbra_{\Wcal',\Wcal} - \kappa \lbra \bfd,\bfuT \rbra_{\Ucal',\Ucal}.
\end{align}
\begin{remark}\label{sign:indefinite}
The quadratic form $\sfU\mapsto G(\sfU,\sfU)$ with $\kappa\shgeq0$ is not sign-definite: $G(\sfU,\sfU)\sheq \Acal(\bfw,\bfw)\shgeq0$ for $\sfU\sheq(\bfw,\bfze)$, and $G(\sfU,\sfU)\sheq -\Acal(\bfw,\bfw)\shm \kappa \Dcal(\bfu,\bfu)\shleq0$ for any $(\bfw,\bfu)$ satisfying~(\ref{stat:pb}a) with $\bff\sheq\bfze$.
\end{remark}
Let $\sfG:\Ubb\to\Ubb'$ be the (bounded) linear operator associated to the bilinear form $G$. The above definition of $G$ implies $\sfG\sheq\sfG\tsup$. The variational problem~\eqref{stat:pb} can be shown to be well-posed by checking the applicability of the closed range theorem (see e.g.~\cite[Thm. 5.11-6]{ciarlet:book}), which here has the form:

\begin{lemma}\label{bounding}
Let $\Ubb$ be a Hilbert space and $\sfG:\Ubb\to\Ubb'$ a bounded linear operator such that $\sfG=\sfG\tsup$. 
$\sfG$ is invertible with bounded inverse if and only if there is a constant $\eta\shg0$ such that $\|\sfG\sfU\|_{\Ubb'}\geq \eta\|\sfU\|_{\Ubb}$ for any $\sfU\shin\Ubb$, in which case $\|\sfG^{-1}\|\leq\eta^{-1}$.
\end{lemma}

The null spaces $\Hcal\shdeq N(B)$ and $\Kcal\shdeq N(B\tsup)$ of $B$ and $B\tsup$ play an important role for studying the well-posedness of problem~\eqref{stat:pb}, so we now characterize them. To this end, let $\Zcal$ denote the subspace of $\Ucal$ such that any $\bfu\shin\Zcal$ solves the homogeneous problem
\begin{equation}
  \Acal(\bfu,\bfuT) - \oo^2 (\rho\bfu,\bfuT) = 0, \quad\text{i.e. } \Bcal(\bfu,\bfuT)=0 \qquad \text{for all }\bfuT\shin\Ucal \label{W:eigen}
\end{equation}
As is well-known, $\Zcal\sheq\{\bfze\}$ unless $\oo$ belongs to the countable set of eigenvalues for problem~\eqref{W:eigen}, in which case $\Zcal$ is finite-dimensional.\enlargethispage*{3ex}

An element $\bfu$ of $\Hcal\shsubs\Ucal$ is characterized by $\Bcal(\bfu,\bfwT)=0$ for all $\bfwT\shin\Wcal$. On applying the first Green identity in $\Acal(\cdot,\cdot)$, the strong form of this variational problem is found to be
\begin{subequations}
\begin{equation}
  -\Div(\CS\dip\bfeps[\bfu]) - \rho\oo^2\bfu =\bfze \;\;\text{in }\OO, \qquad \bft[\bfu]=\bfze \;\; \text{on $\GN$}, \qquad \bfu = \bfze \;\; \text{on $\GD$}. \label{strong:u}
\end{equation}
Similarly, an element $\bfw$ of $\Kcal\shsubs\Wcal$ is characterized by $\Bcal(\bfu,\bfw)=0$ for all $\bfu\shin\Ucal$, i.e.
\begin{equation}
  -\Div(\CS\dip\bfeps[\bfw]) - \rho\oo^2\bfw=\bfze \;\;\text{in }\OO, \qquad \bft[\bfw]=\bfze \;\; \text{on $\GN\shcup\Gc$}, \qquad \bfw = \bfze \;\; \text{on $\GD\shcup\Gc$} \label{strong:ub}
\end{equation}
\end{subequations}
in strong form. Then, two cases arise:

\textsl{Case (i).} If $\Wcal=\Ucal$ (i.e. $\GD\shcup\GN\sheq\G$), problems~\eqref{strong:u} and~\eqref{strong:ub} are identical and coincide with problem~\eqref{W:eigen}; therefore $\Hcal=\Kcal=\Zcal$.

\textsl{Case (ii).} $\Wcal\subsetneq\Ucal$ (i.e. $\Gc\shneq\emptyset$), the situation is completely different as the boundary conditions are undetermined in~\eqref{strong:u} and overdetermined in~\eqref{strong:ub}. In the latter case, homogeneous Dirichlet and Neumann data is simultaneously imposed on $\Gc$, therefore problem~\eqref{strong:ub} has only the trivial solution by virtue of the unique continuation principle~\cite[Corollary]{ang:98}, i.e. $\Kcal=\{\bfze\}$. By contrast, $\Hcal$ now includes forced responses for any excitation applied on $\Gc$ and eigenfunctions when $\oo$ is an eigenvalue for any kind of homogeneous data on $\Gc$ (see problem~\eqref{strong:u}), and is thus infinite-dimensional. We do not attempt to characterize $\Hcal$ more precisely, since in this case we will only use the fact that $\Kcal=\{\bfze\}$.

The following property of the bilinear form $\Bcal$, whose proof is given in Sec.~\ref{proof:infsup}, is crucial for establishing the well-posedness of the stationarity problem~\eqref{stat:pb}:
\begin{lemma}\label{infsup}
There exists $\beta\shg0$ such that the bilinear form $\Bcal:\Ucal\shtimes\Wcal\to\Rbb$ introduced in problem~\eqref{stat:pb} satisfies the inf-sup condition
\begin{equation}
  \inf_{\bfw\in\Kcal{}^{\perp}} \sup_{\bfu\in\Ucal} \frac{\Bcal(\bfu,\bfw)}{\|\bfu\|_{\oo}\|\bfw\|_{\oo}}
 = \inf_{\bfu\in\Hcal^{\perp}} \sup_{\bfw\in\Wcal} \frac{\Bcal(\bfu,\bfw)}{\|\bfu\|_{\oo}\|\bfw\|_{\oo}} = \beta. \label{infsup:b}
\end{equation}
Moreover, let $0\shleq\oo_1\shleq\oo_2\shleq\ldots$ be the eigenvalues associated with the homogeneous problem~\eqref{W:eigen}. 
Then, there exists $0\shl\xi\shleq1$ such that the inf-sup constant $\beta$ satisfies
\[
  \beta \geq \xi\inf_{\oo_n\not=\oo} \frac{|\oo_n^2-\oo^2|}{\oo_n^2+\oo^2}
\]
If either $\Ucal=\Wcal$ or $\Hcal=\{\bfze\}$, the above inequality holds with $\xi=1$.
\end{lemma}

In addition, the following assumptions are made on the bilinear forms $\Dcal$ and $\Acal$:
\begin{Ass}\label{ass:ker}\enlargethispage*{1ex}

The measurement bilinear form $\Dcal$ is coercive on $\Hcal\shtimes\Hcal$: there exists $\delta\shg0$ such that $\Dcal(\bfv,\bfv) \shgeq \delta\|\bfv\|^2_{\oo}$ for all $\bfv\shin\Hcal$.
\end{Ass}
\begin{Ass}\label{ass:W}
The experimental conditions are such that $|\GD\shcup\Gc|\shg0$. The stiffness bilinear form $\Acal$ is therefore coercive on $\Wcal\shtimes\Wcal$.
\end{Ass}
\begin{remark}\label{boundary:data}
Let $\Ncal:=N(D)\subset\Ucal$ be the null space of $D$. Assumption~\ref{ass:ker} implies that $\Hcal\shcap\Ncal=\{\bfze\}$. If it is not verified, there exists $\bfz\shneq\bfze,\ \bfz\shin\Hcal\shcap\Ncal$. Then, $(\bfw,\bfu)=(\bfze,\bfz)$ solves~\eqref{stat:pb} with $\bff\sheq\bfd\sheq\bfze$, i.e. uniqueness fails for the original stationarity problem. Assumption~\ref{ass:ker} is therefore necessary for the well-posedness of problem~\eqref{stat:pb}, and will be seen to be also sufficient. It means that any nontrivial elastodynamic state satisfying the (possibly incomplete) homogeneous boundary conditions on $\GD,\GN$ must register on the measurement apparatus.

For the case of incomplete BCs (for which $\Hcal$ is infinite-dimensional), assumption~\ref{ass:ker} is a stringent requirement, as it cannot be met with a $L^2$ norm on measurement residuals (since $D$ would then be compact for the $\|\cdot\|_{\oo}$ norm). In this case, $\Dcal$ may for example be defined so as to be equivalent to the $H^1(\OO\msub)$ norm. By contrast, for the complete BC case (for which $\Hcal$ is at most finite-dimensional), coercivity on $\Hcal\shtimes\Hcal$ can be achieved with $\Dcal$ defined in terms of a $L^2$ norm.
\end{remark}

Assumption~\ref{ass:W} excludes the case $\GN=\dO$, i.e. pure Neumann boundary conditions.  The resulting coercivity of $\Acal$ over the whole $\Wcal\shtimes\Wcal$ will make the forthcoming analyses simpler and is not detrimental in practice: under usual experimental conditions, the sample will not undergo known surface forces on the whole boundary $\dO$ while being completely ``unsupported."\enlargethispage*{3ex}

We are now in a position to establish the well-posedness of the stationarity problem~\eqref{stat:pb}. To do so, 
still following~\cite{boffi}, we use the decomposition $\Ucal=\Hcal\oplus\Hcal^{\perp}$ and split any $\bfu\shin\Ucal$ and $\bfd\shin\Ucal'$ according to
\begin{equation}
  \bfu = \bfu_0+\bfu_1 \quad (\bfu_0\shin\Hcal,\,\bfu_1\shin\Hcal^{\perp}), \qquad
  \bfd = \bfd_0+\bfd_1 \quad (\bfd_0\shin\Hcal',\bfd_1\shin\lpar\Hcal^{\perp})'), \label{ud:split}
\end{equation}
to cater for $\Hcal$ being potentially non-trivial, noting that $\Bcal(\bfu,\bfwT) = \Bcal(\bfu_1,\bfwT)$. We therefore have
\begin{equation}
  \lbra \bfd,\bfuT \rbra_{\Ucal',\Ucal}
 = \lbra \bfd_0,\bfuT_0 \rbra_{\Ucal',\Ucal} + \lbra \bfd_1,\bfuT_1 \rbra_{\Ucal',\Ucal}. \label{f:kernel}
\end{equation}
Using these definitions and notations, we obtain the following main result, whose proof is given in Section~\ref{wellposed:proof}:
\begin{theorem}
\label{wellposed}
Let Assumptions~\ref{ass:ker} and~\ref{ass:W} hold. Then, for every $\bff\shin\Wcal'$ and $\bfd\shin\Ucal'$ and for any $\kappa\shg0$, problem~\eqref{stat:pb} has a unique solution $(\bfw,\bfu)\in\Wcal\shtimes\Ucal$, that moreover satisfies
\[
  \|\bfu\|+\|\bfw\| \leq C\lpar \|\bff\| + \|\bfd\| \rpar,
\]
where the constant $C$ depends only on $\kappa$, the stability constants $\alpha,\beta,\delta$ and the continuity constants $a,\,d$. More precisely, with reference to the form~(\ref{stat:pb:block},b) of problem~\eqref{stat:pb}, the following estimates hold (with all dependences on $\kappa$ in the continuity constants made explicit):
\begin{align}
  \|\bfw\| &\leq \alpha^{-1}\|\bff\| + Q \|\bfd_0\| + \kappa \beta^{-1}q_a\|\bfd_1\|, \\
  \|\bfu_0\| &\leq \beta^{-1}q_d(1\shp q_a) \|\bff\| + (1\shp q_d r_a Q)\|\bfd_0\| + \kappa \beta^{-1}q_a q_d r_a\|\bfd_1\|, \\
  \|\bfu_1\| &\leq \beta^{-1}(1\shp q_a) \|\bff\| + r_a Q \|\bfd_0\| +\kappa \beta^{-1}q_a r_a\|\bfd_1\|.  
\end{align}
where the non-dimensional constants $q_a,q_d,r_a,r_d$ are defined by $q_a\shdeq a/\alpha$, $q_d\shdeq d/\delta$, $r_a\shdeq a/\beta$ and $r_d\shdeq d/\beta$ and the constant $Q$ is given by $2Q := \kappa\alpha^{-1}\lpar q_d r_a + \sqrt{4\alpha(\kappa\delta)^{-1} + q^2_c r^2_a} \rpar$.
\end{theorem}

\begin{remark}
Consider finite-dimensional subspaces $\Wcal_h \shsubs\Wcal$ and $\Ucal_h \shsubs\Wcal$ (e.g. from finite element discretizations), whose dimension depends on a discretization parameter $h$. If the conditions of Theorem~\ref{wellposed} hold uniformly with $h$ (i.e. for all discretization levels), the  discrete  systems arising from~\eqref{stat:pb} are well-posed.  Moreover, if the approximability property holds (i.e. elements of $\Ucal,\Wcal$ can be approximated arbitrarily closely by elements of $\Ucal_h,\Wcal_h$ for some small enough $h$), then the sequence of solutions of the discrete systems converges (as $h \to 0$) to the solution of~\eqref{stat:pb} in the given norm~\cite{boffi}.
\end{remark}

\subsection{Supplementary assumption for identification feasibility} Assumptions~\ref{ass:ker} and~\ref{ass:W} ensure the well-posedness of the stationarity problem~\eqref{stat:pb}. Additional requirements are however needed to ensure that the data $\bfu\msup$ also provides useful information towards the original elastic imaging inverse problem. To see this, consider the case where $N(D)=:\Ncal\sheq\Hcal^{\perp}$, for which the experimental data is just sufficient to satisfy Assumption~\ref{ass:ker}. In that case, the stationarity problem~\eqref{stat:pb} reads
\begin{equation}
\text{Find }(\bfw,\bfu_1,\bfu_0)\in\Wcal\shtimes\Hcal^{\perp}\shtimes\Hcal, \quad \left\{
\begin{aligned}
  \text{(a)} && \Acal(\bfw,\bfwT) + \Bcal(\bfwT,\bfu_1) &= \lbra \bff, \bfwT \rbra_{\Wcal',\Wcal} && \text{for all }\bfwT\shin\Wcal \\[-0.5ex]
  \text{(b)} && \Bcal(\bfuT_1,\bfw) &= 0 && \text{for all }\bfuT_1\shin\Hcal^{\perp} \\
  \text{(c)} && \Dcal(\bfu_0,\bfuT_0) &= \lbra \bfd_0,\bfuT_0 \rbra_{\Ucal',\Ucal} && \text{for all }\bfuT_0\shin\Hcal
\end{aligned} \right.
\end{equation}
having written equation~(\ref{stat:pb}b) separately for $\bfuT\sheq\bfuT_1\shin\Hcal^{\perp}$ and $\bfuT\sheq\bfuT_0\shin\Hcal$. Equation (c) determines $\bfu_0$ solely from the measurements (so $\bfu_0$ depends neither on $\bff$ nor on the assumed elastic properties $\CS$), while equations (a), (b) show that $\bfw,\bfu_1$ do not depend on the measurements. The case $\Ncal\sheq\Hcal^{\perp}$ is therefore a limiting situation where the measurement $\bfu\msup$ carries no information on $\CS$. We therefore introduce the following additional assumption on the measurement configuration: 
\begin{Ass}\label{ass:N}
$\Ncal$ is a proper subspace of $\Hcal^{\perp}$, i.e. has a nontrivial orthogonal complement in $\Hcal^{\perp}$.\enlargethispage*{3ex}
\end{Ass}

\section{Stationarity solution asymptotics}
\label{sec:3}

To gain insight into the effect of the adjustable weight parameter $\kappa$ on the minimization problem~\eqref{eqn:InvProbDefinition}, we now seek the limiting forms of the solution of the stationarity problem~\eqref{stat:pb} in the limiting situations $\kappa\to0$ and $\kappa\to+\infty$. Consequences of the results of this section, in particular regarding sign properties of the Hessian of the MECE reduced functional, are discussed in Sections~\ref{sec:large-k-obj} and~\ref{sec:small-k-obj}.


We start by recasting problem~\eqref{stat:pb} using the splitting~\eqref{ud:split} and operator notation, to obtain
\begin{subequations}
\begin{equation}
  \sfG \sfU = \sfF \label{stat:pb:block}
\end{equation}
where $\sfU\shin\Ubb:=\Wcal\shtimes\Ucal=\Wcal\shtimes\Hcal\shtimes\Hcal^{\perp}$, $\sfF\shin\Ubb'$ and the operator $\sfG:\Ubb\to\Ubb'$ are defined by
\begin{equation}
  \sfG = \begin{Bmatrix} A & 0 & B \\ 0 & -\kappa D_{00} & -\kappa D_{10} \\ B\tsup & -\kappa D_{01} & -\kappa D_{11} \end{Bmatrix}, \qquad
  \sfU = \begin{cmatrix} \bfw \\ \bfu_0 \\ \bfu_1 \end{cmatrix}, \qquad
  \sfF = \begin{cmatrix} \bff \\ -\kappa\bfd_0 \\ -\kappa\bfd_1 \end{cmatrix}\;. \label{block4:def}
\end{equation}
\end{subequations}
The above block form~(\ref{stat:pb:block},b) of the stationarity problem~\eqref{stat:pb} is such that its first row is equation~(\ref{stat:pb}a) whereas the remaining two rows are equation~(\ref{stat:pb}b) with $\bfuT=\bfuT_0\shin\Hcal$ and $\bfuT=\bfuT_1\shin\Hcal^{\perp}$, in that order. The operators $D_{00},D_{10},D_{01},D_{11}$ are the $\Hcal\to\Hcal'$, $\Hcal^{\perp}\to\Hcal'$, $\Hcal\to(\Hcal^{\perp})^{\prime}$ and $\Hcal^{\perp}\to(\Hcal^{\perp})^{\prime}$ restrictions of $D:\Ucal\to\Ucal'$, respectively (these restrictions being defined by $D_{ij}\sheq P_j DE_i$ in terms of the extension operators $E_0:\Hcal\to\Ucal$, $E_1:\Hcal^{\perp}\to\Ucal$ and the orthogonal projectors $P_0:\Ucal\to\Hcal$, $P_1:\Ucal\to\Hcal^{\perp}$). The zero blocks in~\eqref{block4:def} account for $\Hcal$ being the null space of the operator $B$.\enlargethispage*{3ex}

\subsection{Small-$\kappa$ expansion} We begin by deriving the leading expansion of the stationarity solution $(\bfw,\bfu)=(\bfw^{\kappa},\bfu^{\kappa})$ about $\kappa\sheq0$, assuming \emph{a priori} the expansion to have the form
\begin{equation}
  \bfw^{\kappa} = \bfw^{(0)} + \kappa\bfw^{(1)} + \ldots, \qquad
  \bfu^{\kappa}_0 = \bfu_0^{(0)} + \kappa\bfu_0^{(1)} + \ldots, \qquad
  \bfu^{\kappa}_1 = \bfu_1^{(0)} + \kappa\bfu_1^{(1)} + \ldots \label{ansatz:small}
\end{equation}
(i.e. to follow the format $\sfU_{\kappa}=\sfU^{(0)}+\kappa\sfU^{(1)}+\ldots$), inserting the ansatz~\eqref{ansatz:small} into problem~\eqref{stat:pb:block},~\eqref{block4:def} and writing the resulting $O(1),\,O(\kappa)\ldots$ equations. The $O(1)$ equations are readily obtained as:
\begin{subequations}
\begin{equation}
  \begin{Bmatrix} A & B \\ B\tsup & 0 \end{Bmatrix} \begin{cmatrix} \bfw^{(0)} \\ \bfu^{(0)}_1 \end{cmatrix}
 = \begin{cmatrix} \bff \\ 0 \end{cmatrix} \label{smallkappa:a}
\end{equation}
Thanks to Assumption~\ref{ass:W}, they define a well-posed problem for $\lpar\bfw^{(0)},\bfu^{(0)}_1\rpar$ by~\cite[Cor.~4.2.1]{boffi}, since $A$ is coercive and $B$ is a bounding $\Hcal^{\perp}\to\Wcal$ operator. Then, the $O(\kappa)$ equations are
\begin{equation}
  D_{00}\bfu^{(0)}_0 = \bfd_0 - D_{10}\bfu^{(0)}_1, \label{smallkappa:b}
\end{equation}
which (since $D_{00}$ is coercive by Assumption~\ref{ass:ker}) defines a well-posed problem for $\bfu^{(0)}_0$, and
\begin{equation}
  \begin{Bmatrix} A & B \\ B\tsup & 0 \end{Bmatrix} \begin{cmatrix} \bfw^{(1)} \\ \bfu^{(1)}_1 \end{cmatrix}
 = \begin{cmatrix} 0 \\ -\bfd_1 + D_{01}\bfu^{(0)}_0 + D_{11}\bfu^{(0)}_1 \end{cmatrix}, \label{smallkappa:c}
\end{equation}
which define a well-posed problem for $\lpar\bfw^{(1)},\bfu^{(1)}_1\rpar$. We now can state and prove the following result:
\end{subequations}
\begin{prop}\label{prop:smallkappa}
The stationarity solution admits the small-$\kappa$ expansion $\sfU=\sfU^{(0)} + O(\kappa)$, in the sense of the $\|\cdot\|_{\Ubb}$ norm. The components $\bfw,\bfu_0,\bfu_1$ of $\sfU^{(0)}$ solve the well-posed problems~\eqref{smallkappa:a} and~\eqref{smallkappa:b}. Moreover, we have $\bfw^{(0)}\sheq\bfze$ in the ``usual'' case where $\Kcal=\{\bfze\}$.
\end{prop}
\begin{proof}
Define the truncation error $\Delta\sfU=\{\bfw \shm \bfw^{(0)},\,\bfu_0 \shm \bfu_0^{(0)},\,\bfu_1 \shm \bfu_1^{(0)}\}=\sfU^{\kappa}\shm\sfU^{(0)}$. On setting $\sfU^{\kappa}\sheq\sfU^{(0)}\shp\Delta\sfU$ in~\eqref{stat:pb:block} and using equations~\eqref{smallkappa:a}-\eqref{smallkappa:c}, the governing system for $\Delta\sfU$ is
\begin{equation}
  \sfG\,\Delta\sfU = \sfY, \qquad \text{with} \quad
  \sfY = \begin{cmatrix} \bff \\ -\kappa\bfd_0 \\ -\kappa\bfd_1 \end{cmatrix}
 - \begin{Bmatrix} A & 0 & B \\ 0 & -\kappa D_{00} & -\kappa D_{10} \\ B\tsup & -\kappa D_{01} & -\kappa D_{11} \end{Bmatrix}
   \begin{cmatrix} \bfw^{(0)} \\ \bfu^{(0)}_0 \\ \bfu^{(0)}_1 \end{cmatrix}
 = \begin{cmatrix} \bfze \\ \bfze \\ \kappa B\tsup\bfw^{(1)} \end{cmatrix}\;,
\label{smallkappa:pb}
\end{equation}
and therefore has the form~(\ref{stat:pb:block},b) with $\bff\sheq\bfd_0\sheq\bfze$ and $\bfd_1=-B\tsup\bfw^{(1)}$. Theorem~\ref{wellposed} hence applies to problem~\eqref{smallkappa:pb}, yielding the bounds
\[
  \|\bfu \shm \bfu^{(0)}\| \leq \kappa q_a r_a (q_d\shp1)\beta^{-1} \|B\tsup\bfw^{(1)}\|, \qquad
  \|\bfw \shm \bfw^{(0)}\| \leq \kappa\beta^{-1}q_a \|B\tsup\bfw^{(1)}\|.
\]
Consequently there exists $\kappa_0\shg0$ and a constant $C(\kappa_0)$ such that we have $\|\Delta\sfU\|_{\Ubb}\leq C(\kappa_0)\kappa$ for all $\kappa\shl\kappa_0$. Finally, if $\Kcal\sheq\{\bfze\}$, the second equation of~\eqref{smallkappa:a} implies $\bfw^{(0)}\sheq\bfze$. The proof is complete.
\end{proof}

\subsection{Large-$\kappa$ expansion} We now seek an expansion of $\sfU_{\kappa}$ about $\kappa\sheq+\infty$, assumed to have the form $\sfU_{\kappa}=\sfU^{(0)}\shp\kappa^{-1}\sfU^{(1)}\shp\ldots$. Using this ansatz into the block form~(\ref{stat:pb:block},b) of problem~\eqref{stat:pb}, the arising leading-order $O(\kappa)$ equation is
\begin{equation}
  \begin{Bmatrix} -D_{00} & -D_{10} \\ -D_{01} & -D_{11} \end{Bmatrix}
  \begin{cmatrix} \bfu^{(0)}_0 \\ \bfu^{(0)}_1 \end{cmatrix}
 = \begin{cmatrix} -\bfd_0 \\ -\bfd_1 \end{cmatrix}.
\end{equation}
The chosen ansatz requires that this equation be well-posed. Assumption~\ref{ass:ker} on $\Dcal$, which postulates only the coercivity of $D_{00}$, is insufficient in this respect and must be replaced by a stronger requirement:
\begin{Ass}\label{ass:ker:largekappa}
$\Dcal$ is coercive on $\Ncal^{\perp}\shtimes\Ncal^{\perp}$ (with $\Ncal\sheq N(D)$): there exists $\delta\shg0$ such that $\Dcal(\bfv,\bfv) \shgeq \delta\|\bfv\|^2_{\oo}$ for all $\bfv\shin\Ncal^{\perp}$.
\end{Ass}
This also makes the splitting~\eqref{ud:split} unsuitable for studying the large-$\kappa$ limiting case. We instead split $\bfu$ according to
\begin{equation}
  \bfu = \bfu_0+\bfu_1 \qquad (\bfu_0\shin\Ncal^{\perp},\,\bfu_1\shin\Ncal), \label{u:split:largekappa}
\end{equation}
which implies $\bfd_1\sheq\bfze$. The operator form~\eqref{stat:pb:block} of the coupled problem now uses the definitions
\begin{equation}
  \sfG = \begin{Bmatrix} A & B_0 & B_1 \\ B_0\tsup & -\kappa D & 0 \\ B_1\tsup & 0 & 0 \end{Bmatrix}, \qquad
  \sfU = \begin{cmatrix} \bfw \\ \bfu_0 \\ \bfu_1 \end{cmatrix}, \qquad
  \sfF = \begin{cmatrix} \bff \\ -\kappa\bfd_0 \\ \bfze \end{cmatrix}\;, \label{block5:def}
\end{equation}
based on the splitting~\eqref{u:split:largekappa}, instead of~\eqref{stat:pb:block}. Its solution satisfies (as shown in Sec.~\ref{est:largekappa:proof}) the estimates
\begin{align}
  \|\bfw\|
 &\leq \alpha^{-1} \lpar \|\bff\| + s_d \|\bfd_0\| \rpar, \nnum
  \|\bfu_0\|
 &\leq \kappa^{-1} s_a \delta^{-1} \|\bff\| + \lpar 1 + \kappa^{-1} s_a s_d \rpar \delta^{-1}\|\bfd_0\|, \label{U:largekappa} \\
  \|\bfu_1\|
 &\leq \lpar 1 + q_a + s_a s_d \rpar \beta^{-1}\lpar \|\bff\| + s_d \|\bfd_0\| \rpar \notag
\end{align}
with $s_a \shdeq b/\alpha$ and $s_d \shdeq b/\delta$. Problem~\eqref{stat:pb:block},~\eqref{block5:def} therefore has a unique solution $\sfU$ that depends continuously on the data $\bff,\bfd_0$ (as already known from Theorem~\ref{wellposed}). Moreover, and importantly, all continuity constants in estimates~\eqref{U:largekappa} are bounded in the limit $\kappa\to\infty$. Choosing $\kappa_0\shg0$, we obtain the existence of a constant $C(\kappa_0) \shg0$, independent on $\kappa$, such that $\|\sfU\|_{\Ubb} \leq C(\kappa_0)$ for all $\kappa\shgeq\kappa_0$.

We now proceed by assuming the expansion of the stationarity solution $(\bfw,\bfu)=(\bfw^{\kappa},\bfu^{\kappa})$ about $\kappa\sheq+\infty$ to have, in terms of the new splitting~\eqref{u:split:largekappa}, the form
\begin{equation}
  \bfw^{\kappa} = \bfw^{(0)} + \kappa^{-1}\bfw^{(1)} + \ldots, \qquad
  \bfu^{\kappa}_0 = \bfu_0^{(0)} + \kappa^{-1}\bfu_0^{(1)} + \ldots, \qquad
  \bfu^{\kappa}_1 = \bfu_1^{(0)} + \kappa^{-1}\bfu_1^{(1)} + \ldots. \label{ansatz:large}
\end{equation}
Inserting this ansatz into~\eqref{block5:def} results in $O(\kappa),\, O(1)\ldots$ equations. The $O(\kappa)$ equation is
\begin{equation}
  -D\bfu^{(0)}_0 = -\bfd_0; \label{largekappa:a}
\end{equation}
it is satisfied by $\bfu^{(0)}_0=\bfu\msup$ since the definition of $\bfd$ in~\eqref{stat:pb} and decomposition~\eqref{u:split:largekappa} imply that $\bfd_0=D\bfu\msup$. Then, the $O(1)$ equations yield the two uncoupled well-posed systems
\begin{equation}
  \text{(a) \ }\begin{Bmatrix} A & B_1 \\ B_1\tsup & 0 \end{Bmatrix} \begin{cmatrix} \bfw^{(0)} \\ \bfu^{(0)}_1 \end{cmatrix}
 = \begin{cmatrix} \bff-B_0\bfu\msup \\ \bfze \end{cmatrix}, \qquad
  \text{(b) \ }D\bfu^{(1)}_0 = B\tsup_0\bfw^{(0)}. \label{largekappa:bc}
\end{equation}

\begin{prop}\label{prop:largekappa}
The stationarity solution admits the large-$\kappa$ expansion $\sfU=\sfU^{(0)} + O(\kappa^{-1})$, in the sense of the $\|\cdot\|_{\Ubb}$ norm, with the components $\bfw^{(0)},\bfu^{(0)}_0,\bfu^{(0)}_1$ of $\sfU^{(0)}$ solving the well-posed problems~(\ref{largekappa:bc}a,b).
\end{prop}
\begin{proof}
We examine the truncation error $\Delta\sfU$, defined as for Prop.~\ref{prop:smallkappa}. The governing system for $\Delta\sfU$ still has the format~\eqref{smallkappa:pb}, its right-hand side $\sfY$ being now given, from using~\eqref{largekappa:a} and~(\ref{largekappa:bc}a,b), by
\begin{equation}
  \sfY = \begin{cmatrix} \bff \\ -\kappa\bfd_0 \\ 0 \end{cmatrix}
 - \begin{Bmatrix} A & B_0 & B_1 \\ B_0\tsup & -\kappa D & 0 \\ B_1\tsup & 0 & 0 \end{Bmatrix}
   \begin{cmatrix} \bfw^{(0)} \\ \bfu\msup \\ \bfu^{(0)}_1 \end{cmatrix}
 = \begin{cmatrix} \bfze \\ -B\tsup_0\bfw^{(0)} \\ \bfze \end{cmatrix} \label{largekappa:Y}
\end{equation}
The problem~\eqref{stat:pb:block},~\eqref{largekappa:Y} for $\Delta\sfU$ has the form~\eqref{stat:pb:block},~\eqref{block5:def}, with $\bff\sheq\bfze$ and $\bfd_0=\kappa^{-1}B\tsup_0\bfw^{(0)}$. Its solution therefore satisfies estimates~\eqref{U:largekappa}, which for this particular right-hand side give
\begin{equation}
  \|\Delta\bfw\|
 \leq \frac{s_a s_d}{\kappa} \|\bfw^{(0)}\|, \quad
  \|\Delta\bfu_0\|
 \leq \frac{s_d}{\kappa}\lpar 1 \shp \frac{s_a s_d}{\kappa} \rpar \|\bfw^{(0)}\|, \quad
  \|\Delta\bfu_1\|
 \leq \inv{\kappa}\lpar 1 \shp q_a \shp \frac{s_a s_d}{\kappa} \rpar \frac{s_d}{\beta} \|\bfw^{(0)}\|
\end{equation}
Consequently, there exists $C(\kappa_0) \shg0$ such that $\|\Delta\sfU\|_{\Ubb} \leq \kappa^{-1} C(\kappa_0)$ for all $\kappa\shgeq\kappa_0$.
\end{proof}

\section{Derivatives of the reduced MECE functional}
\label{sec:hessianAnalysis}

In this section, we derive general formulas for the gradient and Hessian of $\LambdaRk(\CS)$ at any $\CS\shin\Qcal$. Then, taking advantage of the results of Sec.~\ref{sec:3}, we show that the Hessian of $\LambdaRk(\CS)$ is asymptotically positive in the large-$\kappa$ case but sign-indefinite in the small-$\kappa$ case connected to least-squares mininimization.

\subsection{First-order derivative of $\LambdaRk$} Recall, from remark (d) of Section~\ref{remarks:coupled}, that the reduced objective $\LambdaRk$ is continuously differentiable with respect to $\CS$.  The first-order derivative is \emph{a priori} given by
\begin{equation}
  \LambdaRk'(\CS) = D_{\Ccal}\Lambdak(\bfu,\bfsig,\CS)[\CSh] = D_{\Ccal}\Lcal(\bfu,\bfw,\bfsig,\CS)[\CSh]
 = \partial_{\Ccal}\Lambdak(\bfu,\bfw,\bfsig,\CS)[\CSh],
\end{equation}
where $D_{\Ccal}$ denotes a total derivative w.r.t. $\CS$ and the prime $(\cdot)^\prime$ symbol is used as a shorthand notation for a derivative w.r.t. $\CS$ in a given direction $\CSh$ (e.g. $\LambdaRk'(\CS)=\LambdaRk'(\CS)[\CSh]$). The last two equalities follow from definition~\eqref{min:C} of $\LambdaRk$ implying verification of the stationarity equations~(\ref{eq:DsigL}-c), and in particular of the balance constraint, for any $\CS$. An explicit expression of $\LambdaRk'(\CS)$ is then obtained as
\begin{equation}
  \LambdaRk'(\CS)
 = \pdemi \lpar \bfeps[\bfu],\CSh\dip\bfeps[\bfu] \rpar - \pdemi\lpar \CS^{-1}\dip\bfsig,\CSh\dip\CS^{-1}\dip\bfsig \rpar
 = \pdemi \Acal(\bfu,\bfu,\CSh) - \pdemi \Acal(\bfu\shp\bfw,\bfu\shp\bfw,\CSh). \label{LambdaR'}
\end{equation}
In particular, any $\CS^{\star}$ solving the reduced minimization problem~\eqref{min:C} must satisfy the first-order optimality condition
\[
  \Acal(\bfu^{\star},\bfu^{\star},\CSh) - \Acal(\bfu^{\star}\shp\bfw^{\star},\bfu^{\star}\shp\bfw^{\star},\CSh) = 0 \qquad
  \text{for all }\CSh \in \Qcal
\]

\subsection{First-order derivative of the stationarity solution} The second-order derivative of $\LambdaRk$ will involve (from applying the total derivative operator $D_{\Ccal}$ to~\eqref{LambdaR'}) the first-order derivative $(\bfu',\bfw')$ of the stationarity solution $(\bfu,\bfw)$ (see remark (d) of Section~\ref{remarks:coupled}). Upon differentiating w.r.t. $\CS$ the stationarity problem~\eqref{stat:pb}, the stationarity solution derivative $(\bfu',\bfw')\in\Ucal\shtimes\Wcal$ solves the variational problem
\begin{equation}
\begin{aligned}
  \text{(a)} && \Acal(\bfw',\bfwT,\CS) + \Bcal(\bfu',\bfwT,\CS) &= -\Acal(\bfw,\bfwT,\CSh) -\Acal(\bfu,\bfwT,\CSh) && \text{for all }\bfwT\shin\Wcal, \\
  \text{(b)} && \Bcal(\bfuT,\bfw',\CS) - \kappa \Dcal(\bfu',\bfuT) &= -\Acal(\bfuT,\bfw,\CSh) && \text{for all }\bfuT\shin\Ucal,
\end{aligned} \label{dd:stat:pb}
\end{equation}
which is well posed since its governing operator is identical to that of the stationarity problem~\eqref{stat:pb}.

\subsection{Second-order derivative of $\LambdaRk$} For convenience, we focus in the sequel on the $\Qcal\to\Rbb$ quadratic form defined by $\CSh\mapsto\LambdaRk''(\CS)[\CSh,\CSh]$, denoted $\LambdaRk''(\CS)$ for short, the corresponding $\Qcal\shtimes\Qcal\to\Rbb$ bilinear mapping being then given by the usual polarization identity $4\LambdaRk''(\CS)[\CSh_1,\CSh_2]=\LambdaRk''(\CS)[\CSh_1\shp\CSh_2,\CSh_1\shp\CSh_2]\shm\LambdaRk''(\CS)[\CSh_1\shm\CSh_2,\CSh_1\shm\CSh_2]$. The second-order derivative $\LambdaRk''(\CS)$ can be derived by differentiating~\eqref{LambdaR'} (e.g. \cite[Thm. 7.8-2]{ciarlet:book}) and rearranging terms, to obtain
\begin{multline}
  \LambdaRk''(\CS)
 = D_{\Ccal}\lpar \LambdaRk'(\CS)[\CSh] \rpar[\CSh]
 = \pdemi D_{\Ccal}\Acal(\bfu,\bfu,\CSh)[\CSh] - \pdemi D_{\Ccal}\Acal(\bfu\shp\bfw,\bfu\shp\bfw,\CSh)[\CSh] \\
 = - \Acal\lpar \bfu, \bfw',\CSh \rpar
    - \Acal\lpar \bfw, \bfu',\CSh \rpar
    - \Acal\lpar \bfw, \bfw',\CSh \rpar. \label{D2L:ED} 
\end{multline}
This expression of $\LambdaRk''(\CS)$ can be recast in a convenient alternative form as follows. Setting $\bfwT\sheq\bfw'$ in~(\ref{dd:stat:pb}a) and $\bfuT\sheq\bfu'$ in~(\ref{dd:stat:pb}b), we find the identity
\begin{equation}
  - \Acal\lpar \bfu, \bfw',\CSh \rpar - \Acal\lpar \bfw, \bfu',\CSh \rpar
  - \Acal\lpar \bfw, \bfw',\CSh \rpar
 = \Acal(\bfw',\bfw',\CS) + 2\Bcal(\bfu',\bfw',\CS)
    - \kappa \Dcal(\bfu',\bfu')
\end{equation}
so that~\eqref{D2L:ED} takes the symmetric form
\begin{equation}
  \LambdaRk''(\CS)
 = \Acal(\bfw',\bfw',\CS) + 2\Bcal(\bfu',\bfw',\CS) - \kappa \Dcal(\bfu',\bfu')
 = G(\sf{U}',\sf{U}',\CS) \label{Lambda'':exp}
\end{equation}

Moreover, $\LambdaRk''(\CS)$ can alternatively be given a sign-revealing form. To this aim, invoking again the splitting~\eqref{u:split:largekappa}, the derivative problem~\eqref{dd:stat:pb} can be written, in operator form, as
\begin{equation}
  \begin{Bmatrix} A & B_0 & B_1 \\ B_0\tsup & -\kappa D & 0 \\ B_1\tsup & 0 & 0 \end{Bmatrix}
  \begin{cmatrix} \bfw' \\ \bfu'_0 \\ \bfu'_1 \end{cmatrix}
 = - \begin{Bmatrix} A' & B'_0 & B'_1 \\ B_0\tsup{}' & 0 & 0 \\ B_1\tsup{}' & 0 & 0 \end{Bmatrix}
  \begin{cmatrix} \bfw \\ \bfu_0 \\ \bfu_1 \end{cmatrix}, \label{dd:stat:pb2}
\end{equation}
where e.g. $A'$ is the operator associated with $\Acal(\cdot,\cdot,\CSh)$.
Using this in~\eqref{Lambda'':exp} produces the following result (whose proof is given in Sec.~\ref{red:hess:proof:sign}), which gives $\LambdaRk''(\CS)$ as an algebraic sum of positive quadratic forms:

\begin{prop}\label{red:hess:sign} Let $(\bfw,\bfu_0,\bfu_1)$ solve the stationarity problem~\eqref{stat:pb:block},~\eqref{block5:def} and $(\bfw',\bfu'_0,\bfu'_1)$ solve the derivative problem~\eqref{dd:stat:pb2}. The second-order derivative of the reduced MECE functional is given by
\begin{equation}
  \LambdaRk''(\CS)
 = \lbra Z\bfw' \shp B_1\bfu'_1,\,\bfw' \shp Z^{-1}B_1\bfu'_1 \rbra_{\Wcal',\Wcal} - \lbra B_1\bfu'_1, Z^{-1}B_1\bfu'_1 \rbra_{\Wcal',\Wcal}
  - \inv{\kappa}\lbra B_0' D^{-1}B_0'{}\tsup \bfw,\bfw \rbra_{\Wcal',\Wcal}
\end{equation}
with the (symmetric, positive, coercive) operator $Z$ defined by $Z\shdeq A\shp\kappa^{-1}B_0D^{-1}B\tsup_0$ in terms of the operators appearing in problem~\eqref{dd:stat:pb2}.
\end{prop}

\subsection{Reduced MECE functional: large-$\kappa$ limiting case}
\label{sec:large-k-obj}

Regarding the leading-order term $(\bfu^{(0)},\bfw^{(0)})$ of the stationarity solution expansion~\eqref{ansatz:large}, we observe that equation~(\ref{largekappa:bc}a) for $(\bfu_1^{(0)},\bfw^{(0)})$ coincides with the stationarity problem for the minimization of the pure ECE functional~\eqref{eqn:ECEdefinition} with the kinematic measurement enforced strictly, the latter being achieved by the solution  $\bfu^{(0)}_0$ of the remaining equation~\eqref{largekappa:a}. Moreover, using expansion~\eqref{ansatz:large} in $\EcalR(\CS),\DcalR(\CS)$ defined by~\eqref{min:C}, we find
\begin{equation}
  \EcalR(\CS) = \pdemi \Acal(\bfw^{(0)},\bfw^{(0)},\CS) + O(\kappa^{-1}), \qquad
  \DcalR(\CS) = \pdemi \kappa^{-2} \Dcal(\bfu^{(1)},\bfu^{(1)}) + O(\kappa^{-2}).
\end{equation}
Consistently with the fact that the data is enforced exactly in the $\kappa\to\infty$ limit, we therefore observe that the value of the reduced objective $\LambdaRk(\CS)$ is for large $\kappa$ dominated by its ECE component.

Moreover, exploiting the large-$\kappa$ expansion of the stationarity solution in the reduced objective $\LambdaRk''(\CS)$ as given by Proposition~\ref{red:hess:sign} for situations where elastic moduli are kept fixed outside the measurement region, $\LambdaRk''(\CS)$ is found to be asymptotically convex in the large-$\kappa$ limit:

\begin{theorem}\label{red:hess} Assume that $\text{supp}(\CSh)\shsubs\OO\msup$ (i.e. that the perturbation $\CSh$ vanishes outside the measurement region), so that $B_1'=0$. Then, the second-order derivative $\LambdaRk''(\CS)$ admits the large-$\kappa$ expansion
\begin{equation}
\LambdaRk''(\CS)
 = \LambdaRk''{}^{(0)}(\CS) + \kappa^{-1}\LambdaRk''{}^{(1)}(\CS) + o(\kappa^{-1}),
\end{equation}
with its leading term $\LambdaRk''{}^{(0)}(\CS)$ given by
\begin{equation}
\LambdaRk''{}^{(0)}(\CS)
 = \lbra (I\shm P_0)\lpar A'\bfw^{(0)}\shp B'_0\bfu\msup \rpar,A^{-1}(I\shm P_0)\lpar A'\bfw^{(0)}\shp B'_0\bfu\msup \rpar \rbra_{\Wcal',\Wcal}.
\end{equation}
In the above expression, the operator $P_0:\Wcal'\to\Wcal'$ is defined by $P_0\shdeq B_1(B_1\tsup A^{-1}B_1^{-1})^{-1}B_1\tsup A^{-1}$ (so $P_0P_0\sheq P_0$, i.e. $P_0$ is a projector, and $A^{-1}P_0\sheq P_0\tsup A^{-1}$), while $\bfw^{(0)}$ is the (leading) zeroth-order term of the large-$\kappa$ expansion of $\bfw$, see Prop.~\ref{prop:largekappa}. The above leading-order coefficient $\LambdaRk''{}^{(0)}(\CS)$ is therefore positive. Moreover, the subsequent coefficient $\LambdaRk''{}^{(1)}(\CS)$ is sign-indefinite.
\end{theorem}
\begin{proof}
See Section~\ref{red:hess:proof}.
\end{proof}

\subsection{Reduced MECE functional: small-$\kappa$ limiting case}
\label{sec:small-k-obj}

We restrict this discussion to the ``usual'' case where $\Kcal=\{\bfze\}$, for which we have $\bfw^{(0)}\sheq\bfze$ (see Prop.~\ref{prop:smallkappa}). Using this and expansion~\eqref{ansatz:small} of the stationarity solution in~\eqref{min:C} and~\eqref{LambdaR'}, the reduced MECE functional $\LambdaRk(\CS)$ and its derivative $\LambdaRk'(\CS)$ have the $O(\kappa)$ expansions
\begin{subequations}
\begin{align}
  \LambdaRk(\CS)
 &= \pdemi\kappa \Dcal(\bfu^{(0)}\shm\bfu\msup,\bfu^{(0)}\shm\bfu\msup) + o(\kappa),
\label{LambdaR:smallkappa} \\
  \LambdaRk'(\CS) &= - \kappa\Acal(\bfu^{(0)},\bfw^{(1)},\CSh) + o(\kappa) \label{LambdaR':smallkappa}
\end{align}
\end{subequations}
Consistently with the constitutive equation is enforced exactly in the $\kappa\to0$ limit, we see that the value~\eqref{LambdaR:smallkappa} of the reduced objective $\LambdaRk(\CS)$ is for small $\kappa$ dominated by its data-misfit component $\Dcal$.

Moreover, equations~(\ref{smallkappa:a},c) give $B\bfu^{(0)}_1\sheq\bff$ and $B\tsup\bfw^{(1)}=-\bfd_1 + D_{01}\bfu^{(0)}_0 \shp D_{11}\bfu^{(0)}_1$. If in addition $\Hcal\sheq\{\bfze\}$ (i.e. boundary conditions are well-posed and $\oo$ is not an eigenvalue for problems~\eqref{W:eigen} and~\eqref{strong:u}), we have that (i) $\bfu_0^{(0)}\sheq\bfze$, (ii) $\bfu_1^{(0)}$ is the forward solution, (iii) $\bfw^{(1)}$ is the adjoint solution for the objective function $\DcalR(\CS)$, (iv) the leading contribution to $\kappa^{-1}\LambdaRk(\CS)$ as $\kappa\to0$ is the reduced quadratic misfit functional $\DcalR(\CS)\shdeq\Dcal(\bfu^{(0)}\shm\bfu\msup,\bfu^{(0)}\shm\bfu\msup)$ commonly used for solving PDE-constrained inverse problems, and (v) the leading term in equation~\eqref{LambdaR':smallkappa} coincides with the known expression for $\DcalR'(\CS)$. Consequently, the MECE-based inversion in reduced form becomes the minimization of the non-regularized least-squares misfit $\Dcal$ in the limit $\kappa\to0$.
Similar remarks apply when $\Hcal\shneq\{\bfze\}$, the (now nontrivial) field $\bfu^{(0)}_0\sheq D_{00}^{-1}(\bfd_0\shm D_{10}\bfu^{(1)})$ being such that the leading-order field $\bfu^{(0)}_0\shp\bfu^{(0)}_1$ minimizes $\Dcal(\bfu^{(0)}\shm\bfu\msup,\bfu^{(0)}\shm\bfu\msup)$ for given $\CS$.\enlargethispage*{3ex}

Then, the stationarity solution expansion~\eqref{ansatz:small} yields expansions $\bfw' = \kappa \bfw^{\prime\,(1)} \shp o(\kappa)$ and $\bfu' = \bfu^{\prime\,(0)} \shp \kappa \bfu^{\prime\,(1)} + o(\kappa)$ for the solution derivatives (since $\bfw^{(0)}\sheq\bfze$, which implies $\bfw^{\prime\,(0)}\sheq\bfze$). Consequently, \eqref{Lambda'':exp} provides
\begin{equation}
  \LambdaRk''(\CS)
 = \kappa \lsqb 2\Bcal(\bfu^{\prime\,(0)},\bfw^{\prime\,(1)},\CS)
   -\Dcal(\bfu^{\prime\,(0)},\bfu^{\prime\,(0)}) \rsqb + o(\kappa), \label{Lambda'':exp:smallkappa:K=0}
\end{equation}
whose leading term is \emph{a priori} sign-indefinite (since it can be recast as an algebraic sum of positive quadratic expressions), by contrast with the corresponding result of Theorem~\ref{red:hess} for the large-$\kappa$ case.\enlargethispage*{3ex}

\section{Numerical results}
\label{sec:num}

In this section, we present numerical studies that support our theoretical findings. We first show (Sec.~\ref{infsup:1D}) that the inf-sup constant remains strictly positive over a wide range of frequencies and displays convergent behavior upon mesh discretization. We then show (Sec.~\ref{sec:convergence}) that finite element discretizations are convergent as long as the two key necessary conditions are met. The next example (Sec.~\ref{sec:convexity}) illustrates how the reduced objective becomes convex as $\kappa$ increases. Finally, we show in Sec.~\ref{sec:imaging} a parameter reconstruction example using interior data, which demonstrates the capability of the method to accommodate underspecified boundary conditions.

\subsection{Stability of coupled system: 1D example}
\label{infsup:1D}

In this example, we study the behavior of the inf-sup constant $\beta$ defined in~\eqref{infsup:b} for the operator $\Bcal$ corresponding to longitudinal vibrations at frequency $f\sheq\oo/2\pi$ of a one-dimensional bar fixed at one end, and whose length, mass density, and Young's modulus are taken as one. The boundary condition at the other end is unspecified. The relevant spaces are $\Wcal := \lcb w\shin H^1(\OO), \; w(0) \sheq w(1) \sheq 0 \rcb$ and $\Ucal := \lcb u\shin H^1(\OO), \; u(0) \sheq0 \rcb$, and we have $\Wcal\subsetneq \Ucal$. The bar is discretized using linear finite elements. 

The inf-sup constant $\beta$ was computed for frequencies $f$ in the $[1,20]$~Hz range, in increments of $0.1$~Hz, the finite element mesh was refined until there was negligible change in $\beta$ over the latter frequency range. The numerical evaluation of $\beta$ uses a discretized version of~\eqref{uhat:def}. The requisite approximation $\bfuH_h$ of the Riesz representative $\bfuH$ of $B\tsup\bfw\in\Ucal'$ is obtained by setting $\bfuH_h\sheq \mathbf{N}_h \bfd$ and $\bfw_h=\mathbf{N}_h \bfv$, where $\mathbf{N}_h$ is a matrix of finite element shape functions for a given discretization with characteristic mesh size $h$. We then take $\mathbf{S}_h$ to be the positive-definite matrix such as $\bfd^T \mathbf{S}_h \bfd := \lpar \bfuH_h,\bfuH_h \rpar_{\oo}$ and $\mathbf{B}_h$ the matrix associated with the bilinear form $\Bcal$ after discretization with finite elements.  Then, the nodal values $\bfd$ associated with $\bfuH_h$  can be computed as $\bfd=\mathbf{S}_h^{-1}\mathbf{B}_h\bfv$. The inf-sup constant is obtained from
\begin{equation}
  \beta_h
 = \inf_{\bfw_h \in \Kcal_h^{\perp}}\frac{ \lpar \bfuH_h,\bfuH_h \rpar_{\oo}}{\| \bfw_h\|^2}
 \ge \inf_{\bfw_h \in \Wcal_h} \frac{\lpar \bfuH_h,\bfuH_h \rpar_{\oo}}{\| \bfw_h\|^2}
 = \inf_{\bfv} \frac{\bfv^T \mathbf{B}_h^T \mathbf{S}_h^{-T}\mathbf{B}_h \bfv}{\bfv^T \mathbf{S}_h \bfv} \geq \min_i \lambda_i,
\end{equation}
(the first inequality following from the fact that $\Kcal_h^{\perp} \subset \Wcal_h$), i.e. by finding the smallest eigenvalue of the generalized eigenvalue problem $\mathbf{B}_h^T \mathbf{S}_h^{-T}\mathbf{B}_h\bfv = \lambda \mathbf{S}_h \bfv$. If that eigenvalue is positive, so is $\beta_h$.

\begin{figure}[t]
\begin{subfigure}[b]{0.49\textwidth}
  \includegraphics[width=0.83\linewidth]{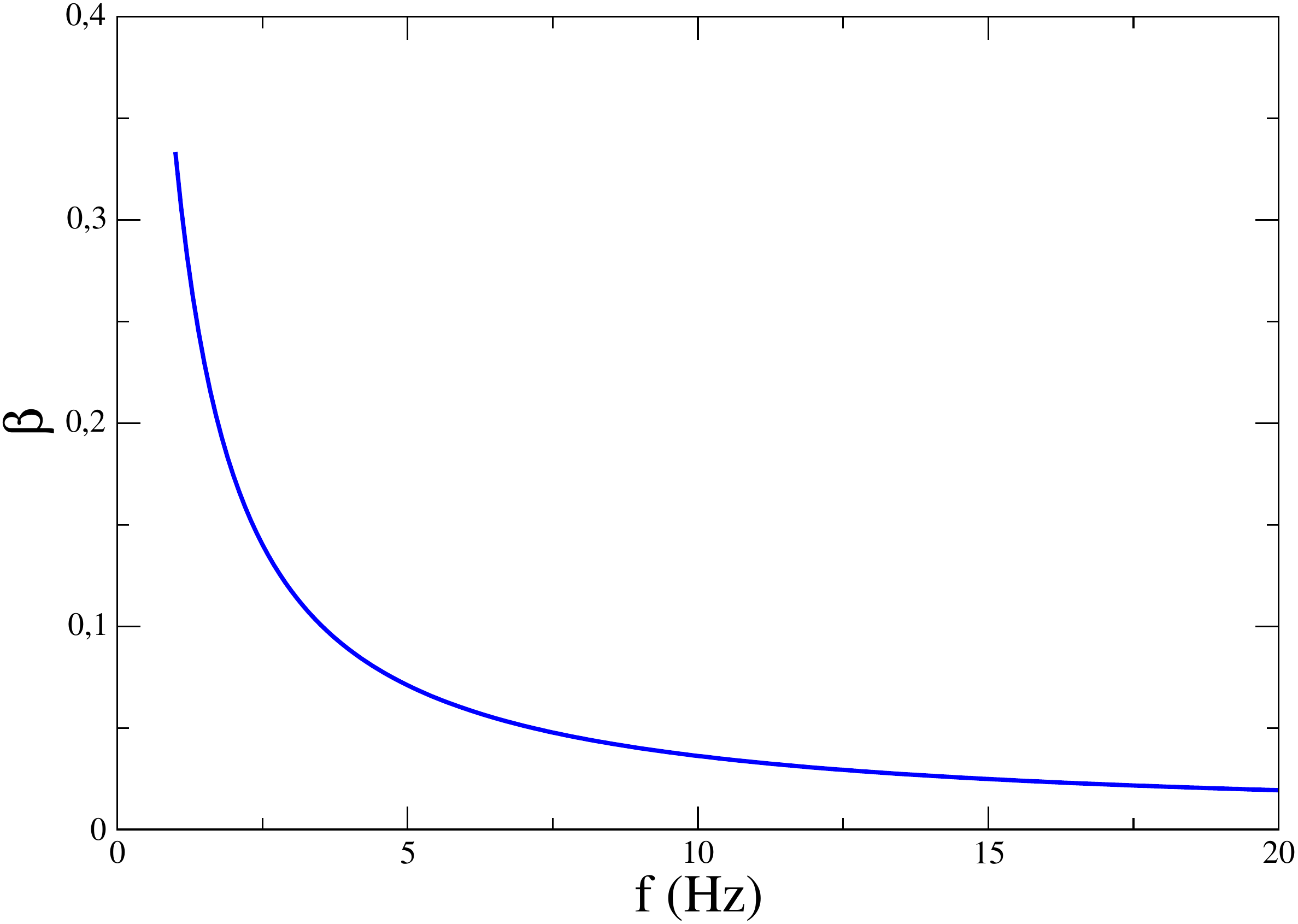}
  \caption{$\beta$ as a function of frequency}
\label{fig:infsup:a}
\end{subfigure}
\begin{subfigure}[b]{0.49\textwidth}
  \includegraphics[width=0.83\linewidth]{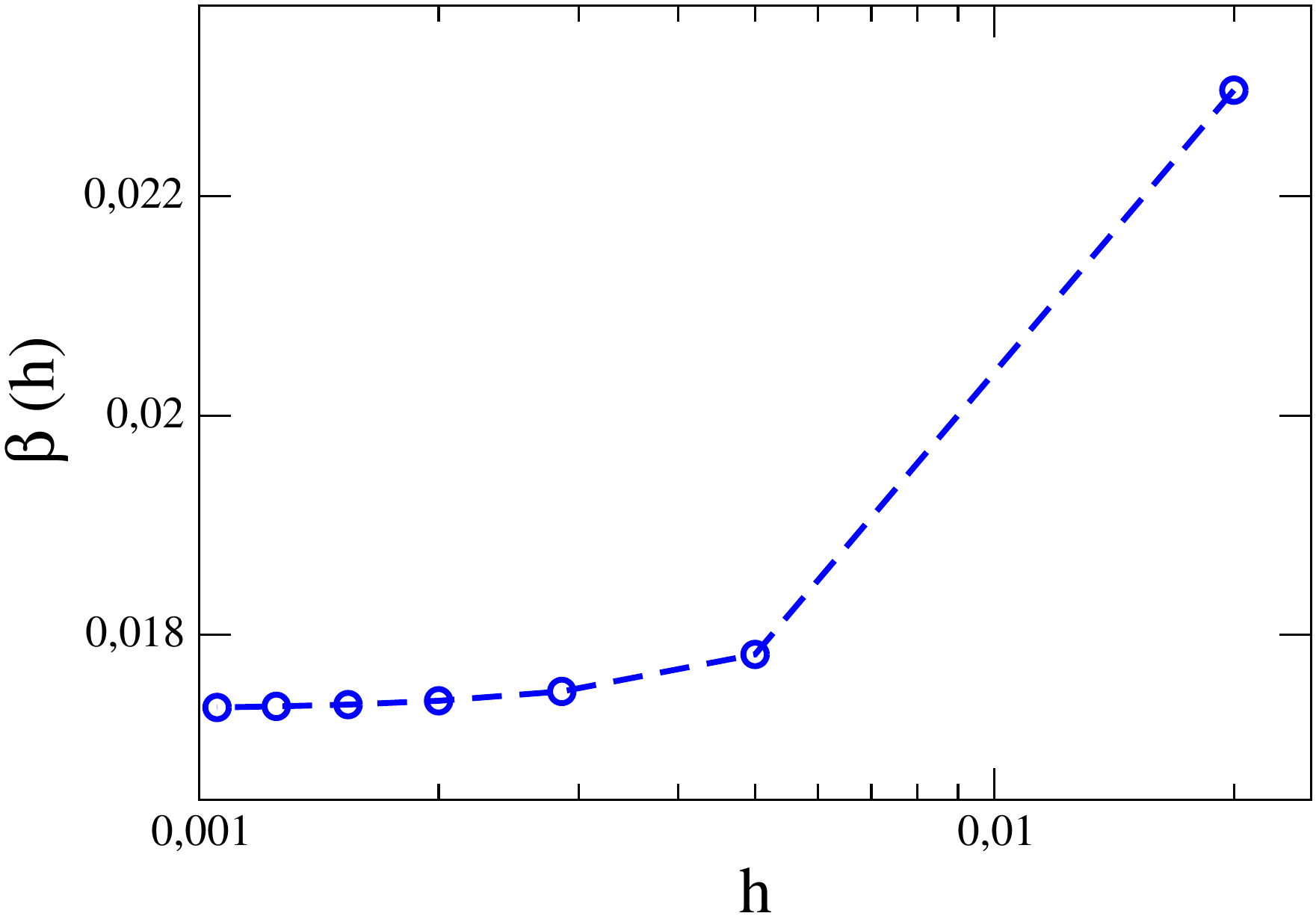}
  \caption{$\beta$ as a function of element size $h$ for $f\sheq 20$~Hz}
\label{fig:infsup:b}
\end{subfigure}
\caption{Inf-sup constant $\beta$ for a one-dimensional bar}
\label{fig:infsup}
\end{figure}

Figure~\ref{fig:infsup:a} shows that indeed $\beta_h$ remains positive for all frequencies investigated in this example. Thus, if the measured data satisfies Assumption~\ref{ass:ker}, the coupled system corresponding to this 1D bar is well-posed in the considered frequency range, in spite of the boundary conditions being underspecified, as expected per Theorem~\ref{wellposed}. Figure~\ref{fig:infsup:b} then shows how $\beta_h$ changes with element size $h$ for a frequency of $20$ Hz, and clearly demonstrates the expected convergence of $\beta$ as the mesh is refined.
  
\begin{figure}[b]
 	\centering
 	\includegraphics[width=0.38\linewidth]{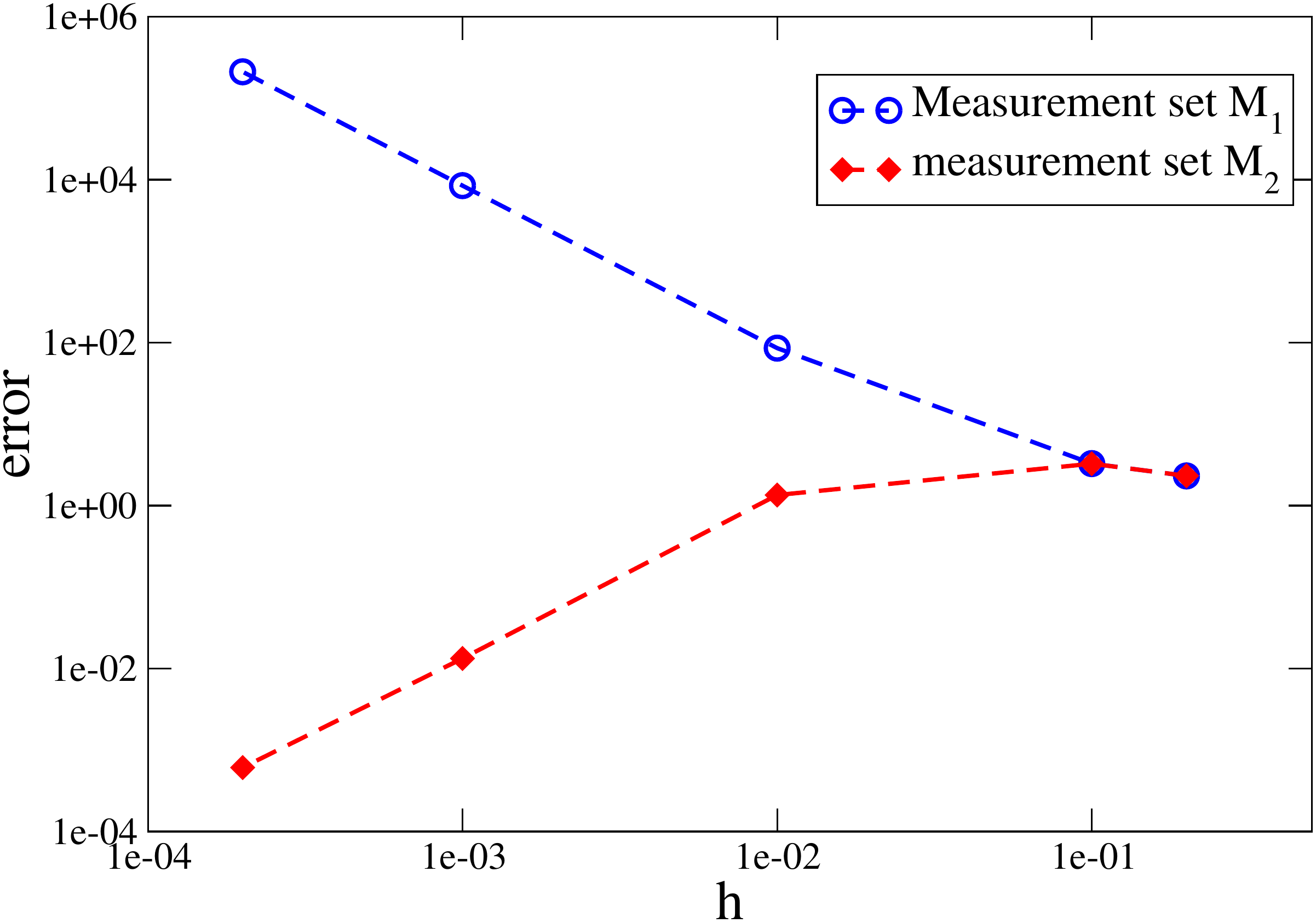}
 	\caption{Convergence of finite element approximations $(u_h,w_h)$ for measurement operators $D_1$ or $D_2$.}
 	\label{fig:convergent}
\end{figure}

\subsection{Well-posedness and convergence of coupled system}
\label{sec:convergence}

In this example, we demonstrate the implications of Assumption~\ref{ass:ker} (i.e. $\Hcal\shcap\Ncal=\{\bfze\}$), which is required for problem~\eqref{stat:pb} to be well-posed, by showing its incidence on the convergence behavior of finite element approximations. We consider again a 1D bar, this time with a Dirichlet condition at both ends, so that we now have $\Wcal = \Ucal = \lcb u\in H^1(\OO), \; u(0) = u(1) = 0 \rcb$. We use an assumed solution $\bar{u} = \sin(\oo x)$ and $\bar{w}= \bar{u}$, which corresponds to an excitation $f= -\oo^2 \sin(\oo x)$ in problem~\eqref{stat:pb}. The observation operator $D$ (as introduced in~\eqref{block4:def}) is of the form $D\bfu=\sum_{k=1}^M u(x_k)$, i.e. involves pointwise measurements at $M$ locations to be specified (pointwise values of elements of $\Ucal$ being well-defined in this 1D setting).\enlargethispage*{3ex}

We study in this example whether finite element discretizations of problem~\eqref{stat:pb} converge and if so at what rate, given that the exact assumed solution $\bar{u}$ is in $\Hcal$. We define (piecewise-linear, conforming) finite element approximations $u_h$ and $w_h$, with $f_h$ being an interpolant of $f$, and consider two versions $D_1$ and $D_2$ of $D$, such that the $M$ measurement locations $x_k$ are taken as $M_1:=\{\frac{\pi}{\oo},\frac{2\pi}{\oo}, \frac{3\pi}{\kappa},...,1 \}$ for $D\sheq D_1$ while $D_2$ uses randomly generated measurement abscissae. Consequently, we have $\bar{u} \in \Ncal$ for $D\sheq D_1$ but $\bar{u} \notin \Ncal$ for $D\sheq D_2$. The case $D\sheq D_1$ thus makes $\Ncal\shcap\Hcal$ non-trivial, in conflict with Assumption~\ref{ass:ker}, whereas $\Ncal\shcap\Hcal=\{0\}$ if $D_2$ is used. We thus expect the coupled system~\eqref{stat:pb} to be ill-posed in the former case, but well-posed in the latter case.

Figure~\ref{fig:convergent} shows the relative error $e_{ls}$ between the manufactured solution $(\bar{u},\bar{w})$ of the coupled system~\eqref{stat:pb} and its finite element approximation $(u_h,w_h)$, defined as
\begin{equation}\label{eq:relativeError}
 e^2_{ls} = \frac{\|\bar{ u} -  u_h\|^2_{L^2(\OO)} + \|\bar{ w} -  w_h\|^2_{L^2(\OO)}}{\|\bar{ u}\|^2_{L^2(\OO)} + \|\bar{ w}\|^2_{L^2(\OO)}},
\end{equation}
as a function of the element size $h$ and for either measurement set. The expected convergence of $(u_h,w_h)$ to $(\bar{u},\bar{w})$ is indeed observed, with the expected $O(h^2)$ rate, if $D=D_2$. By contrast, using $D=D_1$ leads to rapid divergence of $e_{ls}$ as the mesh is refined.

\subsection{Convexity of the reduced MECE functional}
\label{sec:convexity}

We now show an example that demonstrates the convexification of the reduced objective \eqref{min:C} as $\kappa$ increases, predicted by part 2 of Theorem \ref{red:hess}. We consider again a 1D bar problem, this time of the form
\begin{equation}
  -\lpar E(x)u'(x) \rpar'-(2\pi f)^2 u(x) = b \quad (x\in]0,1[), \qquad u(0)=u(1)=0
\label{probl:def:1D:heter}
\end{equation}
with the inhomogeneous Young modulus taken as $E(x)=1_{]0,0.5[}(x)\,E_1 + 1_{]0.5,1[}(x)\,E_2$ and the excitation frequency set to $f\sheq 3$~Hz. The bar is discretized using $100$ linear finite elements. Displacements are assumed to be measured at all nodes. The reduced objective $\LambdaRk(E_1,E_2)$ was computed for $(E_1,E_2)\in[0.5,2.5]^2$ and several values of $\kappa$, by solving the coupled system~\eqref{stat:pb} for each combination $(E_1,E_2)$ and then computing \eqref{min:C}.\enlargethispage*{1ex}

Figures~\ref{fig:convex_1}-\ref{fig:convex_4} illustrate how the reduced objective $(E_1,E_2)\mapsto\LambdaRk(E_1,E_2)$ changes with $\kappa$, and in particular show clearly that this function is progressively smoother as $\kappa$ increases and becomes convex for large $\kappa$ (Figure~\ref{fig:convex_4}). Furthermore, for very small values of $\kappa$ (Figure~\ref{fig:convex_1}) the objective resembles that of a least squares functional, as predicted by the asymptotic analysis of Section \ref{sec:small-k-obj}. 
 
These results have strong practical implications. The least-squares objective function (small-$\kappa$ limiting case) has many local minima, making the inversion results strongly dependent on the initial guess. On the other hand, the MECE objective for intermediate and large $\kappa$ is much smoother and convex, which translates into robustness with respect to the initial guess.
 
It is important to mention that the value of $\kappa$ used for a given problem has to be set according to the level of noise in the data.  Hence, although our analysis indicates that choosing $\kappa$ as large as possible is beneficial for convexity, doing so without regard to noise level would likely produce poor reconstructions. This results from the fact that, as $\kappa$ increases, $\bfu$ becomes closer to the measured data $\bfu\msup$ and may over-fit the noise. See~\cite{B-2013-8} for some discussion on how to choose $\kappa$ according to the Morozov discrepancy principle and a heuristic approach termed error balance.
 
\begin{figure}[h]\label{fig:convexityExample}
 	\centering
 	\begin{subfigure}[b] {0.245\textwidth}
		\includegraphics[width=\textwidth]{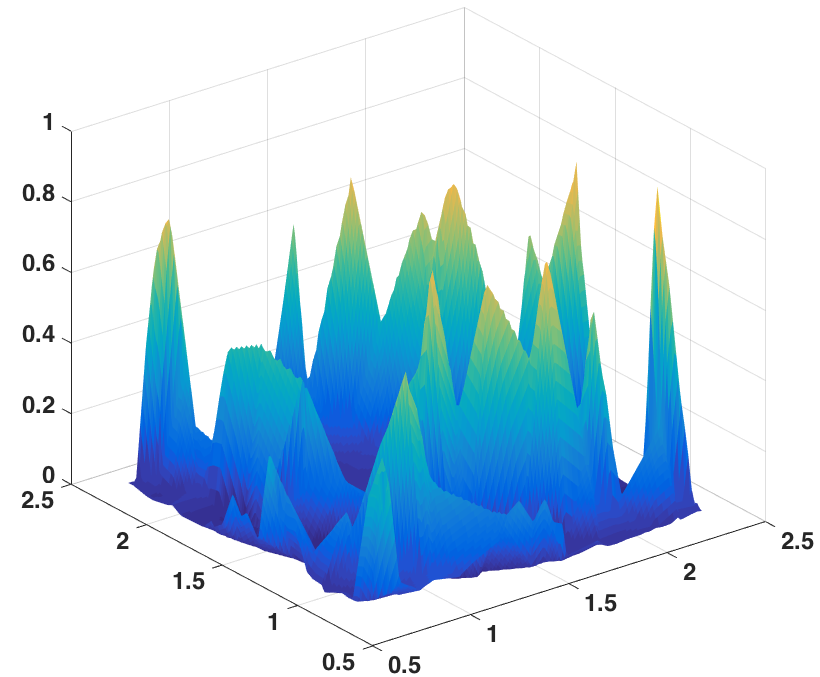}
		\caption{$\kappa = 0.001$}
		\label{fig:convex_1} 
	\end{subfigure}
	\begin{subfigure}[b] {0.245\textwidth}
		\includegraphics[width=\textwidth]{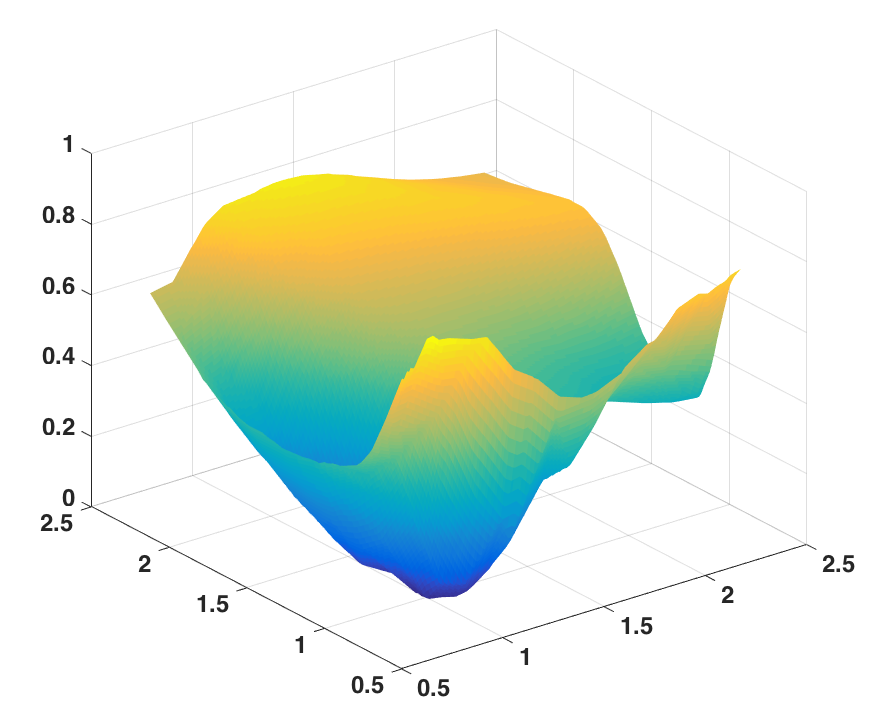}
		\caption{$\kappa = 0.1$}
		\label{fig:convex_2}
	\end{subfigure}
 	\begin{subfigure}[b] {0.245\textwidth}
		\includegraphics[width=\textwidth]{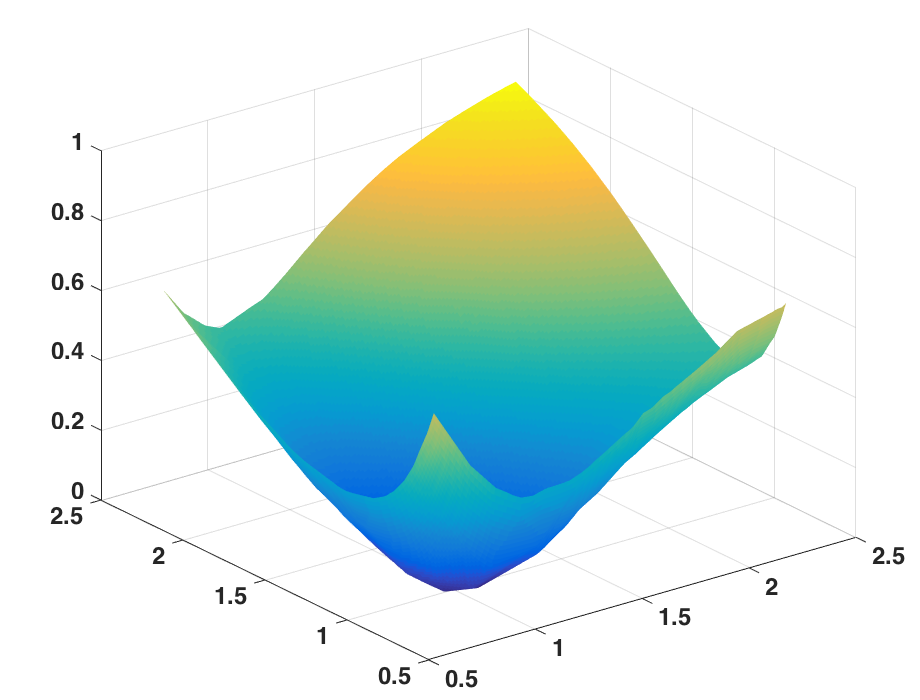}
		\caption{$\kappa = 1$}
		\label{fig:convex_3}
	\end{subfigure}
	\begin{subfigure}[b] {0.245\textwidth}
		\includegraphics[width=\textwidth]{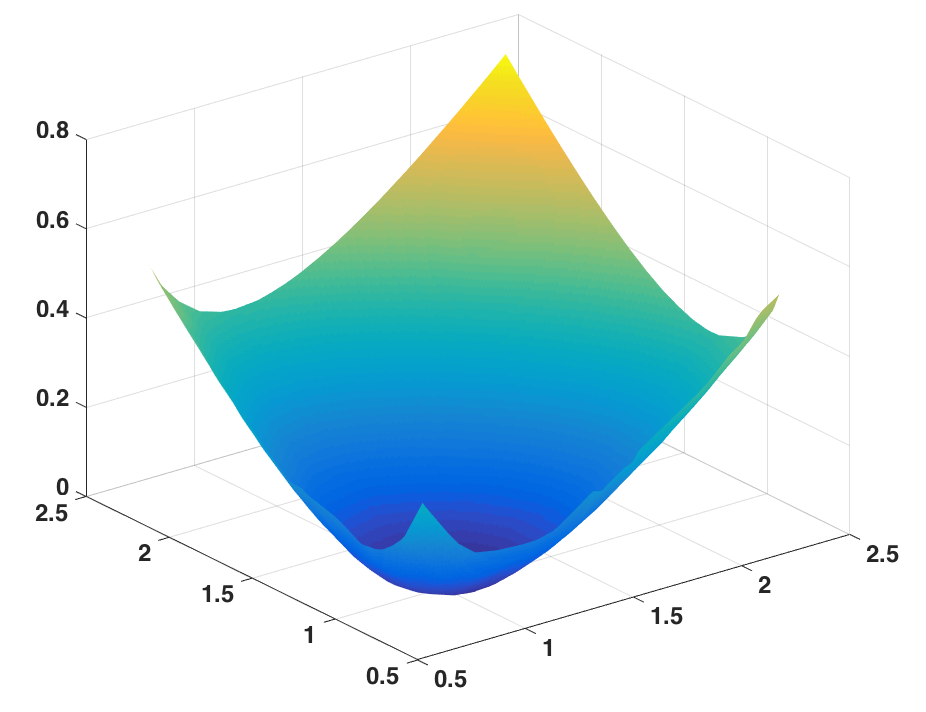}
		\caption{$\kappa = 10^{5}$ }
		\label{fig:convex_4}
	\end{subfigure}
\caption{Reduced objective $\LambdaRk(E_1,E_2)$: convexification with increasing $\kappa$}
\end{figure}
 
\subsection{2D identification example}
\label{sec:imaging}
 
In this last example, we demonstrate that we can successfully reconstruct material properties with MECE when boundary conditions are completely unknown but full field measurements are available in part of the domain. We consider a square elastic domain with an unknown circular inclusion, under 2D plane-strain conditions; see Figure~ \ref{fig:inversion_sketch}. The synthetic experiment consists in loading the body with a uniform time-harmonic pressure of $1$~kPa applied on the top side at a frequency of $10$~Hz, the bottom side being kept fixed, and with the bulk and shear moduli set to $B=8\text{\,kPa},\ G=1.5\text{\,kPa}$ (background) and $B=20\text{\,kPa},\ G=4\text{\,kPa}$ (inclusion). The mass density is uniform, with $\rho=1000\,\text{kg/m}^3$.
Both components of the displacement are assumed to be measured over a dense grid of points in the subdomain $\OO\msup$ (delineated with a dashed line in Figure~\ref{fig:inversion_sketch}). Corresponding synthetic measurements were generated with a fine finite element mesh and interpolated onto a coarser, regular, reconstruction mesh made of $69\shtimes 69$ square 4-noded elements. Each element of that mesh supports two material unknowns $B,G$.  The boundary conditions on all four sides of $\OO\msup$ are taken as unknown.\enlargethispage*{3ex}
 
The reconstruction results shown in Figure~\ref{fig:inversion_results} demonstrate that material properties can be imaged accurately even without the knowledge of boundary conditions; in particular, the recovered values of $B$ and $G$ are close on average to their target values. Interestingly, the reconstructed inclusion displays clear and sharp edges despite the fact that no additional regularization is used for the fields $B,G$.
  
\begin{figure}[t]
\centering
\begin{subfigure}[b]{0.48\textwidth}
\centering
  \includegraphics[width=0.7\linewidth]{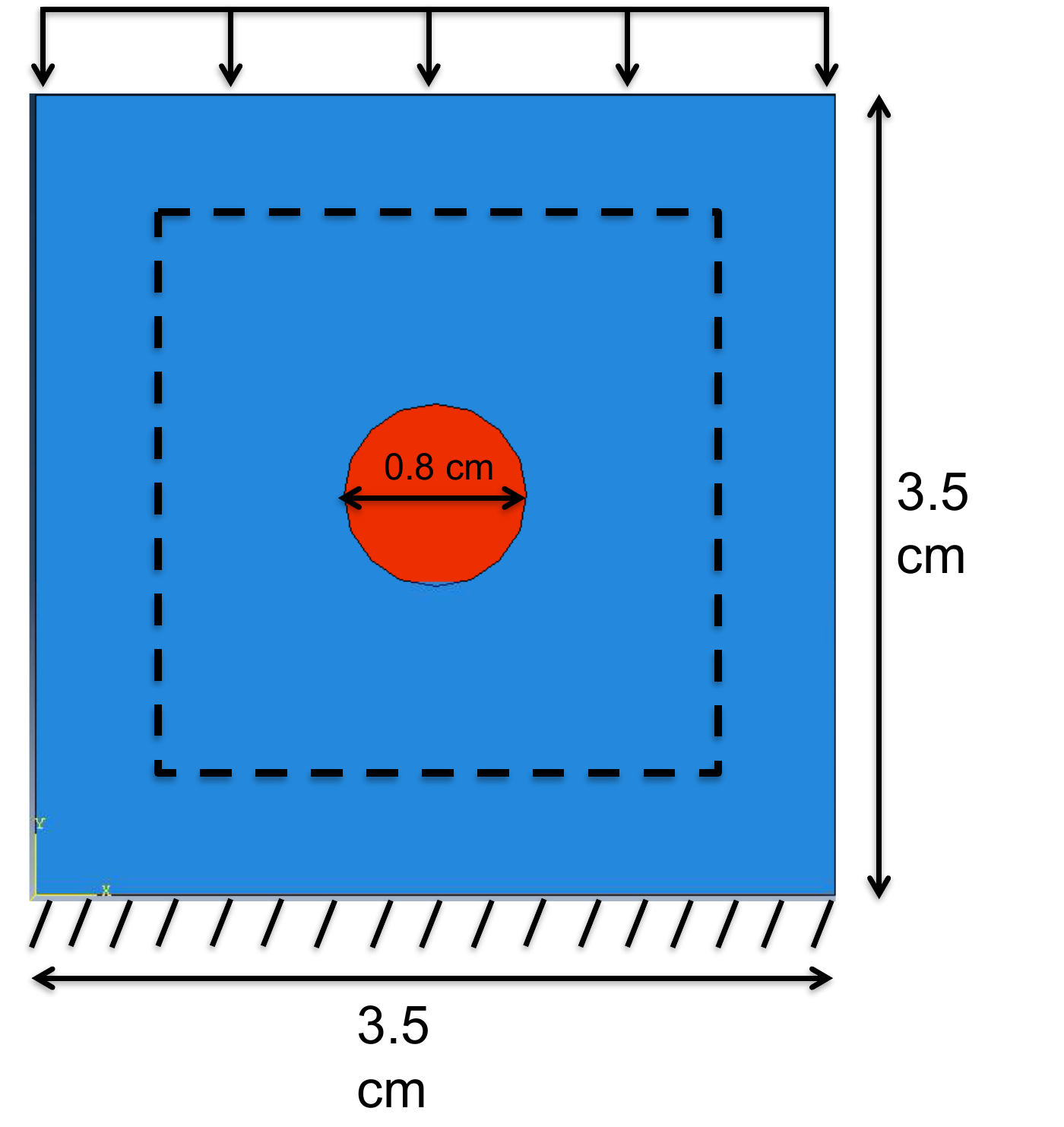}
  \caption{Example configuration and notations\\ \mbox{} \\ \mbox{}}
  \label{fig:inversion_sketch}
\end{subfigure}
\begin{subfigure}[b]{0.48\textwidth}
  \centering
 	\includegraphics[width=0.7\linewidth]{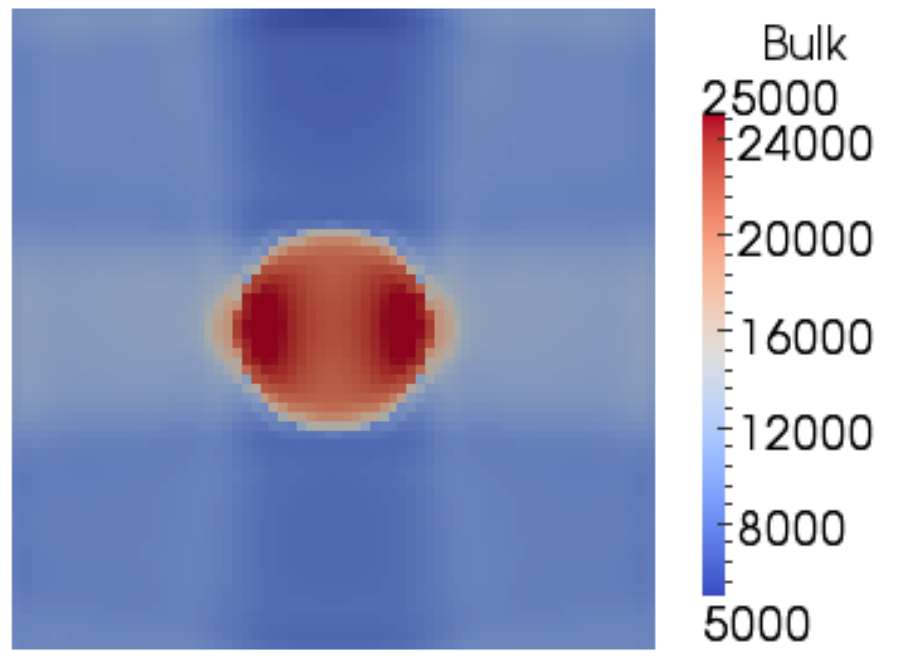}\vspace*{3ex}
 	\includegraphics[width=0.7\linewidth]{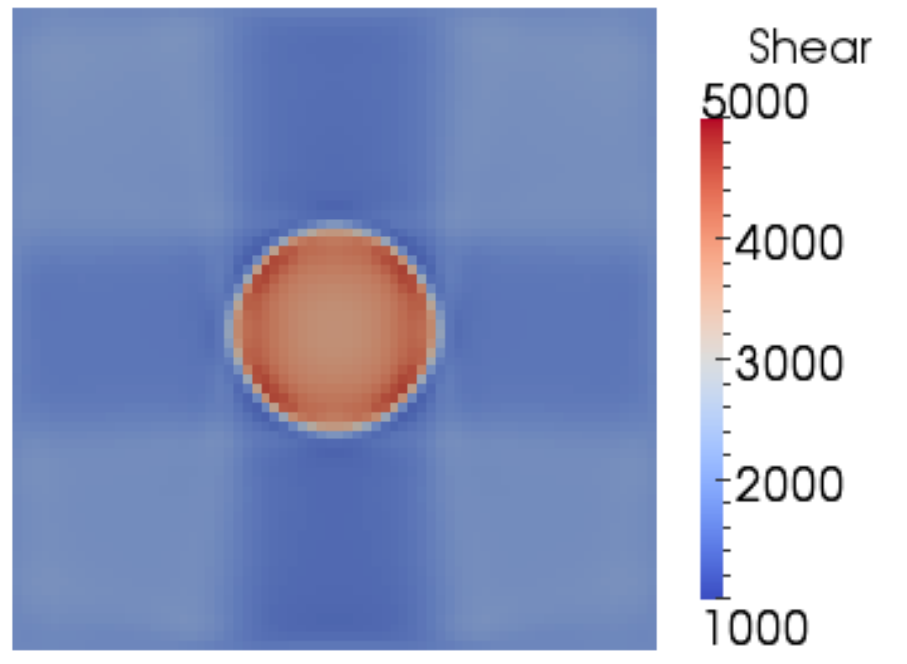}
 	\caption{Results for the reconstructed bulk modulus (top) and shear modulus (bottom) within $\OO\msup$ with unknown boundary conditions.}
 	\label{fig:inversion_results}
\end{subfigure}
 	\caption{2D identification example: definition and results.}
\end{figure}
 
\section{Proofs}
\label{sec:proofs}

\subsection{Proof of Lemma~\ref{infsup}}
\label{proof:infsup}

Let $\bfuH\in\Ucal$ solve the variational problem $\lpar \bfuT,\bfuH \rpar_{\oo}=\Bcal(\bfuT,\bfw)$ for all $\bfuT\shin\Ucal$ (i.e. $\bfuH$ is the $\bfw$-dependent Riesz representant of the linear functional $B\tsup\bfw\in\Ucal'$). We then have
\begin{equation}
  \Bcal(\bfu,\bfw)=\lpar \bfu,\bfuH \rpar_{\oo}, \qquad
  \sup_{\bfu\in\Ucal} \frac{\Bcal(\bfu,\bfw)}{\|\bfu\|} = \| B\tsup\bfw \| = \|\bfuH\|_{\oo}
\label{uhat:def}
\end{equation}

Let $(\bfps_n)_{n\geq1}$ and $(\oo_n)_{n\geq1}$ denote the countable sets of eigenfunctions and eigenvalues associated with problem~\eqref{W:eigen}, i.e. such that each $\bfps_n\shin\Ucal$ verifies $\Acal(\bfps_n,\bfuT) - \oo_n^2 (\rho\bfps_n,\bfuT) = 0$ for all $\bfuT\shin\Ucal$ (we assume for definiteness the normalization $\lpar \rho\bfps_n,\bfps_n \rpar=1$ and the ordering $0\shleq\oo_1\shleq\oo_2\shleq\ldots$); then $(\bfps_n)_{n\geq1}$ is a Hilbert basis of $\Ucal$. Let $I:=\{n\shin\Nbb,\; \oo\sheq\oo_n\}$, so that $\Zcal=\text{span}\lpar \bfps_{n,n\in I} \rpar$, noting that $\Zcal$ is finite-dimensional. With this convention, $I\sheq\emptyset$, i.e. $\Zcal\sheq\{\bfze\}$, if $\oo$ is not an eigenvalue for problem~\eqref{W:eigen}. Expanding the fields $\bfuH\shin\Ucal$ and $\bfw\shin\Wcal$ on the Hilbert basis $(\bfps_n)_{n\geq1}$ (noting for the latter that $\Wcal\shsubs\Ucal$) as
\[
  \bfuH = \sum_{n\geq1} \uhat_n\bfps_n, \qquad \bfw = \sum_{n\geq1} w_n\bfps_n, 
\]
the weak problem linking $\bfuH$ to $\bfw$ then implies (using $\bfuT=\psi_n$ as test functions)
\[
  \uhat_n = w_n \frac{\oo_n^2-\oo^2}{\oo_n^2+\oo^2}.
\]
We therefore obtain
\begin{equation}
  \sup_{\bfu\in\Ucal} \frac{\Bcal(\bfu,\bfw)}{\|\bfu\|} = \|\bfuH\|_{\oo}
   = \Lpar \sum_{n\geq1} \rho(\oo_n^2+\oo^2) |\uhat_n|^2 \Rpar^{1/2}
   = \Lpar \sum_{n\notin I} \rho\frac{(\oo_n^2-\oo^2)^2}{\oo_n^2+\oo^2} |w_n|^2 \Rpar^{1/2} \label{sup:b}
\end{equation}
To establish the existence of an inf-sup constant $\beta$ such that~\eqref{infsup:b} holds, we separately examine two cases: (i) $\Wcal=\Ucal$ and (ii) $\Wcal\subsetneq\Ucal$ (recall that $\Wcal\shsubs\Ucal$).\enlargethispage*{3ex}

\subparagraph{Case (i): $\Wcal=\Ucal$.} This case corresponds to $\G=\GD\shcup\GN$, and we have $\Hcal=\Kcal=\Zcal$. Let $P$ be the orthogonal projection on $\Zcal$, so that
\[
  \sum_{n\notin I} w_n\bfps_n = (I-P)\bfw, \qquad
  \sum_{n\notin I} \rho(\oo_n^2+\oo^2)|w_n|^2 = \|(I-P)\bfw\|^2_{\oo}
\]
Using~\eqref{sup:b}, we have
\begin{equation}
  \Lpar \sup_{\bfu\notin\Ucal} \frac{\Bcal(\bfu,\bfw)}{\|\bfu\|} \Rpar^2
 \geq \Lpar \inf_{n\notin I} \frac{(\oo_n^2-\oo^2)^2}{(\oo_n^2+\oo^2)^2} \Rpar \sum_{n\notin I} \rho(\oo_n^2+\oo^2)|w_n|^2
 = \Lpar \inf_{n\notin I} \frac{(\oo_n^2-\oo^2)^2}{(\oo_n^2+\oo^2)^2} \Rpar \|(I-P)\bfw\|^2_{\oo}. \label{sup:b2}
\end{equation}
Moreover, $\|(I-P)\bfw\|^2_{\oo}=\|\bfw\|^2_{\oo}$ since $\bfw\in\Kcal{}^{\perp}=\Zcal^{\perp}$. Therefore
\begin{equation}
  \sup_{\bfu\in\Ucal} \frac{\Bcal(\bfu,\bfw)}{\|\bfu\|_{\oo}\|\bfw\|_{\oo}}
  \geq \beta, \qquad \beta:=\inf_{n\notin I} \frac{|\oo_n^2-\oo^2|}{\oo_n^2+\oo^2}. \label{sup:b3}
\end{equation}
Since $\beta$ does not depend on $\bfw$, the inf-sup condition~\eqref{infsup:b} holds true with the above value of $\beta$.

\subparagraph{Case (ii): $\Wcal\subsetneq\Ucal$.} In this case, $\Kcal=\{\bfze\}$, so that the infimum in~\eqref{infsup:b} is taken over $\bfw\shin\Wcal$. If $I\sheq\emptyset$, then $\Zcal=\{\bfze\}$ and we again have $\|(I-P)\bfw\|^2_{\oo}=\|\bfw\|^2_{\oo}$; consequently, the argument of case (i) still applies, equations~\eqref{sup:b2} and~\eqref{sup:b3} remain valid (with infimums taken over all integers), and the inf-sup condition~\eqref{infsup:b} again holds with $\beta$ as given by~\eqref{sup:b3}.

On the other hand, a distinct approach is needed for the case $I\shneq\emptyset$ because $\|(I-P)\bfw\|^2_{\oo}=\|\bfw\|^2_{\oo}$ no longer holds for any $\bfw\shin\Wcal$. Instead, we will show below that
\begin{equation}
  \text{There exists } \xi\shg0\; \text{ such that \ } \|(I-P)\bfw\|_{\oo} \geq \xi \|\bfw\|_{\oo} \text{ for all }\bfw\shin\Wcal, \label{min:proj}
\end{equation}
implying
\begin{equation}
  \sup_{\bfu\in\Ucal} \frac{\Bcal(\bfu,\bfw)}{\|\bfu\|_{\oo}\|\bfw\|_{\oo}}
  \geq \beta, \qquad \beta:= \xi\inf_{n\notin I} \frac{|\oo_n^2-\oo^2|}{\oo_n^2+\oo^2}. \label{sup:b4}
\end{equation}
The lemma results from~\eqref{sup:b4}, since that value of $\beta$ does not depend on $\bfw$, provided~\eqref{min:proj} holds.

\paragraph{Proof of~\eqref{min:proj}.} We first note that $\Wcal\shcap\Zcal=\{\bfze\}$: $\bfw\shin\Wcal$ requires $\bfw\sheq\bfze$ on $\Gc$, whereas $\bfw\shin\Zcal$ implies $\bfL(\oo)\bfw\sheq\bfze$ and $\bft[\bfw]\sheq\bfze$ on $\Gc$, and the unique continuation principle implies that only $\bfw\sheq\bfze$ can fulfill all requirements.

The proof then proceeds by contradiction. Assume that~\eqref{min:proj} is false. In that case, we have
\[
  \text{for all } \xi\shg0,\; \text{there exists } \bfw\shin\Wcal,\; \| \bfw-P\bfw\|_{\oo} < \xi \|\bfw\|_{\oo}.
\]
Choose a sequence $(\xi_n)\shg0$ such that $\xi_n\to0$, $n\to\infty$, and for each $\xi_n$ choose $\bfw_n$ such that $\|\bfw_n-P\bfw_n\|_{\oo} < \xi_n \|\bfw_n\|_{\oo}$. $P$ being linear, $\|\bfw_n\|\sheq1$ may be assumed for each $n$ without detriment. Then, $P$ being a projection, $\|P\bfw_n\|\shleq1$: each $P\bfw_n$ belongs to the unit ball $Z=\{\bfz\shin\Zcal,\;\|\bfz\|\shleq1\}$ of $\Zcal$. As $\text{dim}(\Zcal)\shl\infty$, (i) $Z$ is compact, so the sequence $(P\bfw_n)$ contains a convergent subsequence (still denoted $(P\bfw_n)$), whose limit is denoted $\bfz$, and (ii) $\Zcal$ is closed, so $\bfz\shin\Zcal$. On the other hand, we have
\[
  \|\bfw_n-\bfz\|_{\oo} \,=\, \|\,(\bfw_n-P\bfw_n) + (P\bfw_n-\bfz)\,\|_{\oo} \,\leq\, \xi_n + \|P\bfw_n-\bfz\|_{\oo}
\]
implying that $\bfw_n\to\bfz$ as $n\to\infty$. Since $\bfw_n\in\Wcal$ and the Dirichlet trace on $\Gu\shcup\Gc$ is continuous,  $\bfz\in\Wcal$; moreover we have that $\|\bfz\|_{\oo}=1$ (as the limit of a sequence of elements of unit norm), and hence $\bfw\shneq\bfze$. Summarizing, we simultaneously have $\bfw\shin\Wcal\shcap\Zcal$ and $\bfw\shneq\bfze$, which leads to a contradiction because $\Wcal\shcap\Zcal=\{\bfze\}$. Concluding, \eqref{min:proj} is true.

\subsection{Proof of Theorem~\ref{wellposed}}
\label{wellposed:proof}

The proof methods follows that of~\cite[Thm. 4.3.1]{boffi}. We first observe that setting $\bfuT=-\bfu$ and $\bfwT=\bfw$ in~\eqref{stat:pb} and adding the resulting equalities yields $\Acal(\bfw,\bfw) + \kappa \Dcal(\bfu,\bfu) = \kappa\lbra \bfd,\bfu \rbra_{\Ucal',\Ucal} + \lbra \bff,\bfw \rbra_{\Wcal',\Wcal}$
which, by virtue of the assumed coercivity of $\Acal$ on $\Wcal\shtimes\Wcal$, implies the inequality
\begin{equation}
  \alpha \|\bfw\|^2 \leq \|\bff\|\,\|\bfw\| + \kappa\lpar \|\bfd_0\| \|\bfu_0\| + \|\bfd_1\| \|\bfu_1\| \rpar \label{ineq}
\end{equation}
Then, equation~(\ref{stat:pb}b) with $\bfuT=\bfu_0\in\Hcal$ gives
\[
  \Dcal(\bfu_0,\bfu_0) = \lbra \bfd_0,\bfu_0 \rbra_{\Ucal',\Ucal} - \Dcal(\bfu_1,\bfu_0),
\]
which implies the inequality $\delta\|\bfu_0\|^2 \leq \|\bfd_0\| \lpar \|\bfu_0\| + d\,\|\bfu_1\| \rpar$, where the left-hand side results from $\Dcal$ being coercive on $\Hcal\shtimes\Hcal$ by Assumption~\ref{ass:ker}, and the right-hand side from assumed continuity of $\Dcal$ and $\bfd_0$. This results in
\begin{equation}
  \|\bfu_0\| \leq \delta^{-1} \|\bfd_0\| + \delta^{-1} d\,\|\bfu_1\|. \label{ineq:01}
\end{equation}
Finally, using equation~(\ref{stat:pb}a) in operator form (i.e. $A\bfw + B\bfu_1 = \bff$, since $B\bfu_0=\bfze$), we have
\begin{equation}
   \|B\bfu_1\| = \|\bff - A\bfw\| \leq \|A\bfw\| + \|\bff\|,
\end{equation}
which, using the inf-sup condition~\eqref{infsup:b} (which implies that $\|B\bfu_1\|\geq\beta\|\bfu_1\|$ for any $\bfu_1\shin\Hcal^{\perp}$), yields the inequality
\begin{equation}
  \|\bfu_1\| \leq \beta^{-1}a \|\bfw\| + \beta^{-1}\|\bff\|. \label{ineq1}
\end{equation}

We now exploit inequalities~\eqref{ineq}, \eqref{ineq:01} and~\eqref{ineq1}, considering separately each of the three possible cases where only one of $\bff,\bfd_0,\bfd_1$ is nonzero. Considering first the case $\bff\shneq\bfze$, inequalities~\eqref{ineq}, \eqref{ineq:01} and~\eqref{ineq1} readily provide the estimates
\begin{equation}
  \|\bfw\| \leq \alpha^{-1}\|\bff\|, \qquad
  \|\bfu_0\| \leq \beta^{-1}q_d(1\shp q_a) \|\bff\|, \qquad
  \|\bfu_1\| \leq \beta^{-1}(1\shp q_a) \|\bff\|,
\end{equation}
with $q_a,q_d$ as defined in the statement of the proposition.

Consider next the case $\bfd_1\shneq\bfze$. Inequalities~\eqref{ineq}, \eqref{ineq:01} and~\eqref{ineq1} then become
\begin{equation}
  \alpha \|\bfw\|^2 \leq \kappa \|\bfd_1\| \|\bfu_1\|, \qquad
  \|\bfu_0\| \leq q_d \|\bfu_1\|, \qquad
  \|\bfu_1\| \leq r_a \|\bfw\|
\end{equation}
(with $r_a$ as defined in the statement of the proposition), from which we obtain the estimates
\begin{equation}
  \|\bfw\| \leq \kappa \beta^{-1}q_a\|\bfd_1\|, \qquad
  \|\bfu_0\| \leq \kappa \beta^{-1}q_a q_d r_a\|\bfd_1\|, \qquad
  \|\bfu_1\| \leq \kappa \beta^{-1}q_a r_a\|\bfd_1\|.
\end{equation}

Finally, for the case $\bfd_0\shneq\bfze$, inequalities~\eqref{ineq}, \eqref{ineq:01} and~\eqref{ineq1} provide
\begin{equation}
  \alpha \|\bfw\|^2 \leq \kappa \|\bfd_0\| \|\bfu_0\|, \qquad
  \|\bfu_0\| \leq \delta^{-1} \|\bfd_0\| + \delta^{-1} d\,\|\bfu_1\|, \qquad
  \|\bfu_1\| \leq r_a\|\bfw\|,
\end{equation}
Concatenating the above three bounds gives the inequality
\[
  \|\bfw\|^2 - \kappa\alpha^{-1} q_d r_a \|\bfd_0\|\,\|\bfw\| - \kappa\alpha^{-1}\delta^{-1} \|\bfd_0\|^2 \leq 0.
\]
Being of the form $\|\bfw\|^2 \shm A\|\bfw\| \shm B \leq 0$ with $A,B\shg0$, it holds for all $\|\bfw\|\shleq W$ with $W$ the positive root of $w^2 \shm Aw \shm B \sheq 0$, i.e.:
\begin{equation}
  \|\bfw\| \leq Q\|\bfd_0\|, \qquad
  2Q := \kappa\alpha^{-1}\lpar q_d r_a + \sqrt{4\alpha(\kappa\delta)^{-1} + q^2_c r^2_a} \rpar. \label{d0:est1}
\end{equation}
The remaining sought estimates are then
\begin{equation}
  \|\bfu_0\| \leq (1\shp q_d r_a Q) \|\bfd_0\|, \qquad
  \|\bfu_1\| \leq r_a Q\|\bfd_0\|. \label{d0:est2}
\end{equation}
The stationarity problem~\eqref{stat:pb} being linear, the estimates for $\bfu_0,\,\bfu_1,\,\bfw$ for general right-hand sides $\bfd=\bfd_0\shp\bfd_1$ and $\bff$ follow from the triangle inequality.

Finally, since $\Ucal\sheq\Hcal\shoplus\Hcal^{\perp}$, we have $\|\sfF\|_{\Ubb}^2 = \|\kappa\bfd_0\|^2 \shp \|\kappa\bfd_1\|^2 \shp \|\bff\|^2$, implying $\|\kappa\bfd_0\|\shleq \|\sfF\|_{\Ubb}$, $\|\kappa\bfd_1\|\shleq \|\sfF\|_{\Ubb}$ and $\|\bff\|\shleq \|\sfF\|_{\Ubb}$. The obtained estimates of $\|\bfu_0\|$, $\|\bfu_1\|$ and $\|\bfw\|$ therefore collectively imply that there exists a constant $C\shg0$ such that $\|\sfW\|_{\Ubb}\shl C\|\sfG\sfW\|_{\Ubb}$ for any $\sfW\shin\Ubb$. Well-posedness of problem~\eqref{stat:pb} follows by lemma~\ref{bounding} (with $\eta=C^{-1}$).

\subsection{Proof of estimates~\eqref{U:largekappa}}
\label{est:largekappa:proof}

Let $\bfz_0\shin\Ncal^{\perp}$ solve the problem
\begin{equation}
  \kappa\Dcal(\bfz_0,\bfuT_0)
 = \kappa\lbra \bfd_0,\bfuT_0 \rbra \qquad \text{for all }\bfuT_0\shin\Ncal^{\perp}, \label{z0:def}
\end{equation}
which is well-posed ($\Dcal$ being coercive on $\Ncal^{\perp}\shtimes\Ncal^{\perp}$ by Assumption~\ref{ass:ker:largekappa}); moreover, $\bfz_0$ obeys the estimate $\|\bfz_0\|\shleq\delta^{-1}\|\bfd_0\|$. Introducing the new unknown $\bfy\shdeq\bfu\shm\bfz_0$, the coupled problem~\eqref{stat:pb} becomes
\begin{equation}
\begin{aligned}
  \text{(a) \ }&& \Acal(\bfw,\bfwT) + \Bcal(\bfy,\bfwT) &= \lbra \bfh, \bfwT \rbra_{\Wcal',\Wcal} && \text{for all }\bfwT\shin\Wcal, \\
  \text{(b) \ }&& \Bcal(\bfuT,\bfw) - \kappa \Dcal(\bfy,\bfuT) &= \bfze && \text{for all }\bfuT\shin\Ucal.
\end{aligned} \label{stat:pb:largekappa}
\end{equation}
with $\bfh\shdeq\bff\shm B_0\bfz_0$ (since $\bfz_0\shin\Ncal^{\perp}$). Setting $\bfuT=-\bfy$ and $\bfwT=\bfw$ in the above system and adding the resulting equalities yields $\Acal(\bfw,\bfw) + \kappa \Dcal(\bfy,\bfy) = \lbra \bfh,\bfw \rbra_{\Wcal',\Wcal}$, which (by coercivity of $\Acal$ on $\Wcal\shtimes\Wcal$ and applying the triangle inequality to the right-hand side) implies the inequality
\begin{equation}
  \|\bfw\| \leq \alpha^{-1} \|\bfh\| \shleq \alpha^{-1} \lpar \|\bff\| + s_d \|\bfd_0\| \rpar \label{w:largekappa}
\end{equation}
Then, equation~(\ref{stat:pb:largekappa}b) with $\bfuT=\bfu_0\in\Ncal^{\perp}$ gives (in operator form) $\kappa D\bfy_0=B\tsup_0\bfw$, and hence, using~\eqref{w:largekappa}, implies
\begin{equation}
  \|\bfy_0\| \leq \kappa^{-1}\delta^{-1} b\|\bfw\| \leq \kappa^{-1}\delta^{-1} s_a \lpar \|\bff\| + s_d \|\bfd_0\| \rpar.
   \label{v0:largekappa}
\end{equation}
Finally, using equation~(\ref{stat:pb:largekappa}a) in operator form (i.e. $A\bfw \shp B_0\bfy_0 \shp B_1\bfu_1 = \bfh$), we have
\begin{equation}
   \|B_1\bfu_1\| = \|\bfh - A\bfw - B_0\bfy_0 \| \leq \|\bfh\| + \|A\bfw\| + \|B_0\bfy_0\|,
\end{equation}
which, using the inf-sup condition~\eqref{infsup:b} (which provides $\|B_1\bfu_1\|\shgeq\beta\|\bfu_1\|$ for any $\bfu_1\shin\Ncal$ since $\Ncal\shsubs\Hcal^{\perp}$), yields the inequality
\begin{equation}
  \|\bfu_1\|
 \leq \beta^{-1} \lpar a \|\bfw\| \shp b \|\bfy_0\| \shp \|\bfh\| \rpar
 \leq \lpar 1 \shp q_a \shp \kappa^{-1} s_a s_d \rpar \beta^{-1}\lpar \|\bff\| \shp s_d \|\bfd_0\| \rpar.
\end{equation}
Estimates~\eqref{U:largekappa} finally stem from recalling that $\bfu_0\sheq\bfy_0\shp\bfz_0$.

\subsection{Proof of Proposition~\ref{red:hess:sign}}
\label{red:hess:proof:sign}

First, recasting~\eqref{Lambda'':exp} using operator notation, we have
\begin{multline}
  \LambdaRk''(\CS)
 = \lbra A\bfw',\bfw' \rbra_{\Wcal',\Wcal} + 2\lbra B\bfu',\bfw' \rbra_{\Wcal',\Wcal} - \kappa\lbra D\bfu',\bfu'\rbra_{\Ucal',\Ucal} \\
 = \lbra A(\bfw'\shp A^{-1}B\bfu'),\, \bfw'\shp A^{-1}B\bfu' \rbra_{\Wcal',\Wcal}
 - \lbra B\bfu', A^{-1}B\bfu' \rbra_{\Wcal',\Wcal} - \kappa\lbra D\bfu_0',\bfu_0'\rbra_{\Ucal',\Ucal} \label{Lambda'':exp2}
\end{multline}
We then observe that solving the derivative problem~\eqref{dd:stat:pb2} gives 
\begin{align}
  \bfu'_0 &= \kappa^{-1}D^{-1}\lpar B\tsup_0{}'\bfw \shp B_0\tsup\bfw' \rpar, \notag\\
  \bfw' &= Z^{-1} (\bfF \shm B_1\bfu'_1), \label{U':exp} \\
  \bfu'_1 &= (B_1\tsup Z^{-1}B_1)^{-1} (B_1\tsup Z^{-1}\bfF \shm \bfG \rpar, \notag
\end{align}
with the operator $Z:\Wcal\to\Wcal'$ as given in the proposition statement by $Z\shdeq A\shp\kappa^{-1}B_0D^{-1}B\tsup_0$ and having set $\bfH := -A'\bfw - B'_0\bfu_0 - B'_1\bfu_1$, $\bfF :=\bfH-\kappa^{-1}B_0 D^{-1}B_0'{}\tsup\bfw$, $\bfG := -B_1\tsup{}'\bfw$. Using~\eqref{U':exp} and noting that $\kappa^{-1}B_0D^{-1}B\tsup_0\sheq Z\shm A$, we obtain the identities
\begin{align}
  B\bfu'
 &= (Z\shm A)Z^{-1}(\bfF \shm B_1\bfu'_1) + \bfH \shm \bfF + B_1\bfu'_1
  = \bfH - AZ^{-1}(\bfF \shm B_1\bfu'_1) \\
  \bfw'\shp A^{-1}B\bfu'
 &= A^{-1}\bfH \\
  \kappa\lbra D\bfu_0',\bfu_0'\rbra_{\Ucal',\Ucal}
 &= \lbra (Z-A)\bfw',\bfw'\rbra_{\Wcal',\Wcal} + 2\lbra \bfH\shm\bfF,\bfw' \rbra_{\Wcal',\Wcal} + \kappa^{-1}\lbra \Delta \bfw,\bfw \rbra_{\Wcal',\Wcal} \\
 &= \lbra (I \shm AZ^{-1})(\bfF \shm B_1\bfu'_1),\, Z^{-1}(\bfF \shm B_1\bfu'_1) \rbra_{\Wcal',\Wcal} \suite
  + 2\lbra \bfH\shm\bfF,Z^{-1}(\bfF \shm B_1\bfu'_1) \rbra_{\Wcal',\Wcal} + \kappa^{-1}\lbra \Delta\bfw,\bfw \rbra_{\Wcal',\Wcal},
\end{align}
with the symmetric, positive operator $\Delta:\Wcal\to\Wcal'$ given by $\Delta\shdeq B_0' D^{-1}B_0'{}\tsup$. Upon substitution into~\eqref{Lambda'':exp2}, the above formulas provide
\begin{align}
  \LambdaRk''(\CS)
 &= \lbra \bfH,\, A^{-1}\bfH \rbra_{\Wcal',\Wcal}
  - \lbra AZ^{-1}(\bfF \shm B_1\bfu'_1),\,Z^{-1}(\bfF \shm B_1\bfu'_1) \rbra_{\Wcal',\Wcal}
  - \lbra \bfH,\, A^{-1}\bfH \rbra_{\Wcal',\Wcal} \suite
  + 2 \lbra \bfH,\, Z^{-1}(\bfF \shm B_1\bfu'_1) \rbra_{\Wcal',\Wcal}
  - \lbra (I \shm AZ^{-1})(\bfF \shm B_1\bfu'_1),\, Z^{-1}(\bfF \shm B_1\bfu'_1) \rbra_{\Wcal',\Wcal} \suite
  - 2\lbra \bfE,Z^{-1}(\bfF \shm B_1\bfu'_1) \rbra_{\Wcal',\Wcal} - \kappa^{-1}\lbra \Delta \bfw,\bfw \rbra_{\Wcal',\Wcal} \notag \\
 &= \lbra \bfF \shp B_1\bfu'_1,\,Z^{-1}(\bfF \shm B_1\bfu'_1) \rbra_{\Wcal',\Wcal} - \kappa^{-1}\lbra \Delta \bfw,\bfw \rbra_{\Wcal',\Wcal} \notag \\
 &= \lbra \bfF,Z^{-1}\bfF \rbra_{\Wcal',\Wcal} - \lbra B_1\bfu'_1, Z^{-1}B_1\bfu'_1 \rbra_{\Wcal',\Wcal}
  - \kappa^{-1}\lbra \Delta \bfw,\bfw \rbra_{\Wcal',\Wcal} \label{Lambda'':exp3}
\end{align}
Since~\eqref{U':exp} implies $\bfF\sheq Z\bfw' \shp B_1\bfu'_1$, the above expression of $\LambdaRk''$ yields the claimed formula.

\subsection{Proof of Theorem~\ref{red:hess}}
\label{red:hess:proof}

To begin, the expression~\eqref{U':exp} of $\bfu_1'$ gives
\begin{equation}
  B_1\bfu'_1 = P (\bfH \shm \bfE) - B_1(B_1\tsup Z^{-1}B_1)^{-1} \bfG,
\end{equation}
with the operator $P:\Wcal'\to\Wcal'$ defined by $P \shdeq B_1(B_1\tsup Z^{-1}B_1)^{-1} B_1\tsup Z^{-1}$. It is easy to see that $P$ verifies $PP\sheq P$ (i.e. $P$ is a projection) and $Z^{-1}P \sheq P\tsup Z^{-1}$. 

For the case where $B_1'=0$, we have $\bfG\sheq\bfze$, and therefore $B_1\bfu'_1 = P\bfF$. The expression~\eqref{Lambda'':exp3} of $\LambdaRk''(\CS)$ then becomes (with the above definition of $P$)
\begin{align}
  \LambdaRk''(\CS)
 &= \lbra (I \shp P)\bfF,Z^{-1}(I \shm P)\bfF \rbra_{\Wcal',\Wcal}
   - \kappa^{-1}\lbra \Delta \bfw,\bfw \rbra_{\Wcal',\Wcal} \notag\\
 &= \lbra (I \shm P)\bfF,Z^{-1}(I \shm P)\bfF \rbra_{\Wcal',\Wcal}
 - \kappa^{-1}\lbra \Delta \bfw,\bfw \rbra_{\Wcal',\Wcal} \label{Lambda'':exp4}
\end{align}
Next, applying the Sherman-Morrison-Woodbury formula to $Z=A\shp\kappa^{-1}B_0D^{-1}B\tsup_0$ gives
\begin{align}
  Z^{-1} &= A^{-1} - A^{-1}B_0 Q^{-1} B_0\tsup A^{-1} &\qquad \text{with \ }
  Q &:=\kappa D + B_0 A^{-1} B_0\tsup., \\
  (B_1\tsup Z^{-1}B_1)^{-1}
 &= A^{-1}_{11} + A^{-1}_{11}A_{10}R^{-1}A_{01}A^{-1}_{11} &\qquad \text{with \ }
  R &:= Q - A_{01}A^{-1}_{11}A_{10},
\end{align}
wherein $A_{11}\shdeq B_1\tsup A^{-1} B_1$, $A_{01}\shdeq B_0\tsup A^{-1} B_1$, $A_{10}\shdeq B_1\tsup A^{-1} B_0$. Using these identities and performing straightforward algebra, we find
\begin{align}
  I \shm P
 &= \lpar I + P_0 B_0 R^{-1} B_0\tsup A^{-1} \rpar (I\shm P_0) \\
  (I \shm P\tsup) Z^{-1} (I \shm P)
 &= (I\shm P_0)\tsup \lpar A^{-1} - A^{-1}B_0 R^{-1}B_0\tsup A^{-1} \rpar (I\shm P_0)
\end{align}
with the (projection) operator $P_0$ as defined in the theorem statement. We also note that $R^{-1}=\kappa^{-1}D^{-1}\shp o(\kappa^{-1})$, so that
\begin{equation}
  (I \shm P\tsup) Z^{-1} (I \shm P)
 = (I\shm P_0)\tsup \lpar A^{-1} - \kappa^{-1} A^{-1}B_0 D^{-1} B_0\tsup A^{-1} \rpar (I\shm P_0) + o(\kappa^{-1})
\end{equation}
We now substitute the above expansion into~\eqref{Lambda'':exp4} and set $\bfF=\bfF^{(0)}\shp\kappa^{-1}\bfF^{(1)}+o(\kappa^{-1})$ (noting that $\bfF^{(0)}=-A'\bfw^{(0)}-B'_0\bfu\msup$) and $\bfw=\bfw^{(0)}\shp\kappa^{-1}\bfw^{(1)}+o(\kappa^{-1})$, to obtain
\begin{equation}
  \LambdaRk''(\CS)
 = \LambdaRk''{}^{(0)}(\CS) + \kappa^{-1}\LambdaRk''{}^{(1)}(\CS) + o(\kappa^{-1}) \label{Lambda'':expansion}
\end{equation}
with
\begin{align}
  \LambdaRk''{}^{(0)}(\CS)
 &= \lbra (I\shm P_0)\lpar A'\bfw^{(0)}\shp B'_0\bfu\msup \rpar,A^{-1}(I\shm P_0)\lpar A'\bfw^{(0)}\shp B'_0\bfu\msup \rpar \rbra_{\Wcal',\Wcal}, \\
  \LambdaRk''{}^{(1)}(\CS)
 &= 2 \lbra (I\shm P_0)\bfF^{(0)},A^{-1}(I\shm P_0)\bfF^{(1)} \rbra_{\Wcal',\Wcal}
  - \lbra \Delta \bfw^{(0)},\bfw^{(0)} \rbra_{\Wcal',\Wcal} \\
 &= \lbra (I\shm P_0)(\bfF^{(0)}\shp\bfF^{(1)}),A^{-1}(I\shm P_0)(\bfF^{(0)}\shp\bfF^{(1)}) \rbra_{\Wcal',\Wcal} - \lbra \Delta \bfw^{(0)},\bfw^{(0)} \rbra_{\Wcal',\Wcal} \suite\qquad
  - \lbra (I\shm P_0)\bfF^{(0)},A^{-1}(I\shm P_0)\bfF^{(0)} \rbra_{\Wcal',\Wcal}
  - \lbra (I\shm P_0)\bfF^{(1)},A^{-1}(I\shm P_0)\bfF^{(1)}) \rbra_{\Wcal',\Wcal}
\end{align}
Concluding, the above value of $\LambdaRk''{}^{(0)}$ is that claimed in the theorem statement, and is clearly positive, whereas $\LambdaRk''{}^{(1)}$ is obtained as an algebraic sum of positive quadratic expressions (so is \emph{a priori} sign-indefinite). The proof of Theorem~\ref{red:hess} is complete.

\section{Conclusions}

In this work we studied some of the most salient mathematical properties of the Modified Error in Constitutive Equations (MECE) approach for inverse problems in the context of frequency-domain elastodynamics. In particular, we proved (under conditions on the available interior data) well-posedness of the coupled system that arises as part of the first order optimality conditions. The coupled problem remains well posed even when boundary conditions are partially or completely unknown and at resonant frequencies. The latter findings have strong practical implications in inverse problems where interior data is abundant and boundary conditions are difficult to ascertain. We have exploited this benefit in our recent work in elastography \cite{ghosh:17}.

We also showed that the reduced MECE functional becomes convex in the limit where the weight given to the data misfit component goes to infinity.  Moreover, convexification of the reduced objective occurs continuously as the parameter increases, as demonstrated in the numerical examples. This characteristic of the MECE functional also has strong practical implications as solutions to the inverse problem are less sensitive to initial guesses. This fact has also been exploited in our work on elasticity and viscoelasticity imaging \cite{B-2015-1,ghosh:17}.
Future work includes exploring the possibility of devising a unified MECE formulation that can be applied to a wide range of physics such as nonlinear elasticity, electromagnetism, plasticity, fluid dynamics, etc.

\end{document}